\documentclass[10pt,a4paper,oneside,final]{amsart}

\usepackage[backend=bibtex, style=numeric, maxnames=50]{biblatex}
\renewbibmacro{in:}{}

\DeclareFontFamily{OT1}{pzc}{}
\DeclareFontShape{OT1}{pzc}{m}{it}{<-> s * [1.10] pzcmi7t}{}
\DeclareMathAlphabet{\mathpzc}{OT1}{pzc}{m}{it}
\allowdisplaybreaks

\usepackage[T1]{fontenc}
\usepackage[utf8]{inputenc}
\usepackage[dvipsnames]{xcolor}
\usepackage[english]{babel}
\usepackage{mathtools, amsthm, amssymb, amscd} 
\usepackage{thmtools}
\usepackage{newtxtext, newtxmath} 
\usepackage[DIV=11]{typearea} 
\usepackage{hyperref} 
\usepackage[notref,notcite]{showkeys} 
\usepackage[capitalise]{cleveref}
\usepackage{todonotes}
\usepackage{enumitem} 
\usepackage{xypic} 
\usepackage{tikz}
\usetikzlibrary{cd, decorations.markings, arrows} 

\DeclareMathAlphabet{\mathcal}{OMS}{cmsy}{m}{n} 

\definecolor{DarkPurple}{rgb}{0.40,0.0,0.20}
\hypersetup{
	colorlinks =true,
	linkcolor= DarkPurple, 
	citecolor = NavyBlue 
} 

\usetikzlibrary{cd, decorations.markings, arrows, matrix, calc}

\newcommand{\G}{\mathcal{G}}

\newcommand{\Cfull}{C^*}
\newcommand{\Cred}{C_{r}^{*}}

\newcommand{\C}{\mathbb{C}}
\newcommand{\R}{\mathbb{R}}

\newcommand{\N}{\mathbb{N}}

\DeclareMathOperator{\supp}{supp}

\newtheorem{lemma}{Lemma}[section]

\newtheorem{theorem}[lemma]{Theorem}
\newtheorem{proposition}[lemma]{Proposition}

\theoremstyle{definition}
\newtheorem{definition}[lemma]{Definition}
\newtheorem{example}[lemma]{Example}
\newtheorem{remark}[lemma]{Remark}

\title[]{Exotic $\rm C^*$-completions of Étale Groupoids}

\date{}

\author[1]{Mathias Palmstrøm}

\address{Department of Mathematical Sciences, Faculty of Information Technology and Electrical Engineering, NTNU -- Norwegian University of Science and Technology, Trondheim, Norway}
\email{mathias.palmstrom@ntnu.no}


\numberwithin{equation}{section} 

\begin{document}

\renewcommand{\thefootnote}{}

\maketitle

\begin{abstract}
	We generalize the ideal completions of countable discrete groups, as introduced by Brown and Guentner, to second-countable Hausdorff étale groupoids. Specifically, to every pair consisting of an algebraic ideal in the algebra of bounded Borel functions on the groupoid and a non-empty family of quasi-invariant measures on the unit space, we construct a $\rm C^*$-algebra in a way which naturally encapsulates the constructions of the full and reduced groupoid $\rm C^*$-algebras. We investigate the connection between these constructions and the Haagerup property, and use the construction to show the existence of many exotic groupoid $\rm C^*$-algebras for certain classes of groupoids.
\end{abstract}

\section{Introduction}
There are two natural $\rm C^*$-algebras one can associate to an étale groupoid $\G$, the full and the reduced groupoid $\rm C^*$-algebras, often denoted by $C^\ast (\G)$ and $C_{r}^{\ast}(\G)$ respectively. For these, there is a canonical surjection $$ C^{\ast}(\G) \twoheadrightarrow C_{r}^{\ast}(\G) ,$$ induced by the identity on $C_c (\G)$, and the groupoid is said to have the weak containment property when this surjection is an isomorphism \cite[Definition 6.1]{Delaroche:AmenabilityExactnessAndWeakContainmentPropertyForGroupoids}. If $\G$ does not have the weak containment property, then a natural question is whether or not there exist groupoid $\rm C^*$-algebras lying ``in between'' these two; that is, if there exists a $\rm C^*$-norm $\Vert\cdot\Vert_{\mathrm{e}}$ on $C_c (\G)$ which dominates the reduced norm and differs from both the reduced and full norms. In this case, the $\rm C^*$-completion $C_{\mathrm{e}}^{\ast} (\G)$ is called an exotic groupoid $\rm C^*$-algebra, and there are canonical surjections $$ C^{\ast}(\G) \twoheadrightarrow C_{\mathrm{e}}^{\ast} (\G) \twoheadrightarrow C_{r}^{\ast}(\G) ,$$ which are not injective. When $\G$ is a group, this question goes all the way back to \cite{Eymard:FourierAlgebraOfALocallyCompactGroup} wherein such intermediate $\rm C^*$-algebras were considered. However, it was not until Brown and Guentner introduced the so-called ideal completions in \cite{BrownandGuentner:NewC*-completions} that they were systematically studied. Brown and Guentner's construction takes an algebraic ideal $D \trianglelefteq \ell^\infty (\G)$ and defines an associated $\rm C^*$-algebra $C_{D}^{\ast}(\G)$ defined to be the separation and completion of the group ring $\C [\G]$ under the $\rm C^*$-seminorm $\|f\|_D := \sup_\pi \|\pi (f)\|$, where the suprema is taken over all unitary representations of $\G$ which have ``sufficiently many'' matrix coefficients lying inside $D$. They investigated the connection of these completions with amenability, the Haagerup property and property (T), among other things. Regarding the existence of exotic group $\rm C^*$-algebras, they showed that there exists $p \in (2, \infty)$ such that $C_{\ell^p}^{\ast} (\mathbb{F}_d)$ is an exotic group $\rm C^*$-algebra; here $\mathbb{F}_d$ is the non-abelian free group on $2 \leq d < \infty$ generators. Using Brown and Guentner's construction, this result was improved by Okayasu in \cite{Okayasu:FreeGroupC*-algebrasAssociatedWithLP} where he characterized the positive definite functions on $\mathbb{F}_d$ which extend to $C_{\ell^p }^{\ast} (\mathbb{F}_d)$, obtaining as a corollary that the group $\rm C^*$-algebras $C_{\ell^p }^{\ast}(\mathbb{F}_d) $, for $p \in (2, \infty)$, are all distinct. In \cite{Wiersma:ConstructionsOfExoticGroupC*Algebras}, Wiersma extended this result to discrete groups containing a non-abelian free group as a subgroup. The construction introduced by Brown and Guentner generalizes readily to locally compact groups, and we refer the reader to \cite{BekkaKainuthLauSchlichting:OnCstarAlgebarsAssociatedWithLCGroups,  deLaatEtAl:GroupC*AlgebrasOfLocallyCompactGroupsActingOnTrees, KaliszewskiLandstadQuigg:ExoticGroupCstarAlgebrasInNonCommutativeDuality, deLaatAndSiebenand:ExoticGroupC*AlgebrasOfSimpleLieGroupsWithRealRankOne, RuanAndWiersma:OnExoticGroupC*Algebras, SameiAndWiersma:ExoticC*AlgebrasOfGeometricGroups} for results regarding exotic group $\rm C^*$-algebras in the locally compact case. There has also been work on exotic quantum groups (see \cite{BrannanAndRuan:L_PRepresentationsOfDiscreteQuantumGroups, KyedAndSoltan:PropertyTAndExoticQuantumGroupNorms}) and exotic crossed product $\rm C^*$-algebras (see \cite{AntoninietAl:StrongNovikovConjecture..., Baum:ExpanderExactCrossedProductsAndTheBaumConnesConjecture, BussEtAl:ExoticCrossedProducts, BussetAL:ExoticCrossedProductsAndTheBaumConnesConjecture, ExelPittsZarikian:ExoticIdealsInFreeTransformationGroupC*Algebras}), the latter of which turn out to be related to the Baum-Connes and Novikov conjectures. 

For groupoids in general, exotic groupoid $\rm C^*$-algebras have been studied in \cite{BruceAndLi:AlgebraicActionsI.C*-algebrasAndGroupoids, ChristensenAndNeshveyev:IsotropyFibersOfIdealsInGroupoidCstarAlgebras, ChristensenAndNeshveyev:NonExoticCompletionsOfTheGroupAlgebrasOfIsotropyGroups, KwasniewskiAndMeyer:Aperiodicity:theAlmostExtensionPropertyAndUniquenssOfPseudoExpectations, KwasniewskiAndMeyer:EssentialCrossedProductsForInverseSemigroupActions, NeshveyevAndSchwartz:NonHausdorffÉtaleGroupoidsAndC*Algebras, Palmstrom:ExoticCstarAlgsAssocWithDoubleGroupoids}. 
In \cite{BruceAndLi:AlgebraicActionsI.C*-algebrasAndGroupoids}, Bruce and Li considered groupoids associated to algebraic actions of semigroups on groups, and found sufficient conditions for when the concrete $\rm C^*$-algebra generated by the Koopman representation for the action together with the left regular representation of the group is an exotic groupoid $\rm C^*$-algebra.
In \cite{ChristensenAndNeshveyev:IsotropyFibersOfIdealsInGroupoidCstarAlgebras}, Christensen and Neshveyev related the ideal structure of a (exotic) groupoid $\rm C^*$-algebra to that of the associated (exotic) isotropy group $\rm C^*$-algebras. In \cite{Exel:OnKumjiansC*DiagonalsAndTheOpaqueIdeal}, Exel characterized the exotic twisted groupoid $\rm C^*$-algebras of principal twisted étale groupoids as pairs of inclusions $(A,B)$ where $B$ is a $\rm C^*$-algebra and $A$ is a closed commutative $*$-subalgebra of $B$ which is regular and satisfies the extension property. A similar result is given in \cite[Corollary 3.11.7]{ExelAndPitts:CharacterizingGroupoidC*AlgebrasOfNonHausdorffEtaleGroupoids}. Note that the definition of an exotic groupoid $\rm C^*$-algebra there includes the full and reduced $\rm C^*$-algebras. 

Apart the work mentioned above, not much can be found in the literature on exotic groupoid $\rm C^*$-algebras, particularly regarding the question of their existence. This paper aims to be a starting point of investigation into the existence of such groupoid $\rm C^*$-algebras, and is organized as follows: \cref{sec: Prelim} reviews the necessary definitions and results on Hilbert bundles, groupoids and their unitary representations. In \cref{sec: constructing grouopid c*-algebras}, we generalize the construction of ideal completions from \cite{BrownandGuentner:NewC*-completions} to second-countable Hausdorff étale groupoids, which at the very least provides a point of attack in constructing exotic groupoid $\rm C^*$-algebras. \cref{sec: Haagerup property and positive definite functions} is dedicated to investigating the relationship between this construction and the notion of Haagerup property for étale groupoids defined in \cite{KwasniewskiLiSkalski:TheHaagerupPropertyForTwistedGrouopidDynamicalSystems} (see also \cite{Delaroche:HaagerupPropertyForMeasuredGroupoids} and \cite{Tu:LaConjectureDeBaumConnes}). For that, we rely much on the technology for positive definite functions developed by Ramsay and Walter in \cite{RamsayAndWalter:FourierStieltjesAlgebrasOfLocallyCompactGroupoids}. Finally, in \cref{sec: exotic groupoid C*-algebras associated to hyperbolic groupoids}, we use this construction to exhibit some examples of exotic groupoid $\rm C^*$-algebras associated with certain metrically hyperbolic groupoids (see \cref{def: hyperbolic groupoids}). The construction has also been used by the author in \cite{Palmstrom:ExoticCstarAlgsAssocWithDoubleGroupoids} to prove existence of exotic groupoid $\rm C^*$-algebras associated with double groupoids, and we believe that it can be applied to other classes of groupoids, providing yet further examples.

\section*{Acknowledgments}
I am very grateful to Matthew Wiersma for pointing out to me how \cref{thm: there exists extoic ideal completions of a hyperbolic groupoid with appropriate growth conditions} could be improved to its present form. I would also like to thank Eduard Ortega for some insightful conversations and many helpful comments on the first draft of this paper.

\section{Preliminaries} \label{sec: Prelim} 
In this section, we recall the basic notions and results regarding Borel Hilbert bundles, étale groupoids and their unitary representations. There are many excellent references for these subjects, see for example \cite{Dixmier:VonNeumannAlgebras, Muhly:CoordinatesInOperatorAlgebras, Paterson:GroupoidsAndTheirOperatorAlgebras, Renault:AGroupoidApproach, Sims:EtaleGroupoids, Williams:AToolKitForGroupoidCstarAlgebras}. For  \cref{subsec: Borel Hilbert Bundles} and \cref{subsec: unitary repr. of gpds}, we will mostly follow \cite[Appendix F]{Williams:CrossedProductsOfC*Algebras} and \cite[Chapter 3]{Muhly:CoordinatesInOperatorAlgebras}.

\subsection{Borel Hilbert bundles} \label{subsec: Borel Hilbert Bundles}
Throughout this subsection, $X$ will always denote a standard Borel space. 
Let $\{ H(x) \}_{x \in X}$ be a collection of Hilbert spaces indexed by $X$. The associated \emph{Hilbert bundle} is the set $$X \ast H := \{ (x,h) \colon h \in H(x) \}. $$ Let $p \colon X \ast H \to X$ be the map given by $p(x,h) = x ,$ for $(x,h) \in X \ast H$. A \emph{section} of the Hilbert bundle $X \ast H$ is a map $\tilde{f} \colon X \to X \ast H$ such that $p \circ \tilde{f} (x) = x$, for all $x \in X$, and is necessarily of the form $\tilde{f}(x) = (x, f(x))$, for some unique map $f$ with $f(x) \in H(x)$, for all $x \in X$. We will identify the section $\tilde{f}$ with the map $f$.

\begin{definition} \label{def: Borel Hilbert bundle (standard)}
	Let $\{H(x)\}_{x \in X}$ be a collection of Hilbert spaces indexed over $X$, and let $p \colon X \ast H \to X$ be the map as above. We say that the Hilbert bundle $X \ast H$ is a \emph{Borel Hilbert bundle} if $X \ast H$ has a standard Borel structure such that 
	
	\begin{itemize}
		\item[(i)] $E \subset X$ is Borel if and only if $p^{-1}(E)$ is Borel in $X \ast H$;
		\item[(ii)] there is a sequence of sections $\{f_n \}_n$, called a \emph{fundamental sequence}, such that
		\begin{itemize}
			\item[(a)] for each $n$, the map $\hat{f}_n \colon X \ast H \to \C$, given by $\hat{f}_n (x,h) = \langle f_n (x) , h \rangle_{H(x)}$, are Borel;
			\item[(b)] for each $n, m$, the map $x \mapsto \langle f_n (x) , f_m (x) \rangle_{H(x)}$, is Borel;
			\item[(c)] the sequence of functions $\{\hat{f}_n\}_n$ together with $p$, separate points of $X \ast H$.
		\end{itemize}
	\end{itemize} 
\end{definition}

The following lemma proves the usefulness of the fundamental sequence in determining if a section is Borel or not.
\begin{lemma} \label{lem: measurability of maps into Borel hilbert bundle} \cite[Proposition 3.34]{Williams:AToolKitForGroupoidCstarAlgebras}
	Let $X \ast H$ be a Borel Hilbert bundle with fundamental sequence $\{f_n \}_n$. Then 
	\begin{itemize}
		\item[(i)] a section $f \colon X \to X \ast H$ is Borel if and only if the map $x \mapsto \langle f(x) , f_n (x) \rangle_{H(x)}$ is Borel for all $n$;
		\item[(ii)] if $f,g$ are Borel sections, then $x \mapsto \langle f(x), g(x) \rangle_{H(x)}$ is Borel; in particular, $x \mapsto \| f(x) \|_{H(x)}$ is Borel for all Borel sections $f$. 
	\end{itemize} 
\end{lemma}
Property (ii) (c) of \cref{def: Borel Hilbert bundle (standard)} implies that for each $x \in X$, the sequence $\{f_n (x)\}_n$ spans a dense subspace of $H(x)$. It is well-known that one may ``orthonormalize'' the fundamental sequence $\{f_n\}_n$ in such a way that the resulting sequence $\{\tilde{f}_n\}_n$ is still a fundamental sequence for the Hilbert bundle, but also with the property that for each $x \in X$, the nonzero vectors in $\{\tilde{f}_n (x)\}_n$ form an orthonormal basis for $H(x)$. In that case, we shall refer to the fundamental sequence $\{\tilde{f}_n\}_n$ as \emph{orthonormal}. 

The following result is \cite[Proposition 3.2]{Muhly:CoordinatesInOperatorAlgebras} and follows from arguments in \cite[p. 265]{Ramsay:VirtualGroupsAndGroupActions} and \cite[Section 1]{Ramsay:NonTransitiveQuasiOrbitsInMackeysAnalysisOfGroupExtensiond}. We record it here for future reference.
\begin{proposition} \cite[Proposition 3.2]{Muhly:CoordinatesInOperatorAlgebras} \label{prop: family of sections defining standard Borel structure on the Hilbert bundle PaperA}
	Let $X$ be a standard Borel space, and let $X \ast H$ be an associated Hilbert bundle. Suppose we are given a sequence of sections $\{f_n\}_n$ for $X \ast H$ satisfying (ii) (b) and (ii) (c) in \cref{def: Borel Hilbert bundle (standard)}. Then there is a unique Borel structure on $X \ast H$ for which it becomes a Borel Hilbert bundle with $\{f_n\}_n$ as a fundamental sequence.
\end{proposition}

Given a Borel Hilbert bundle $X \ast H$, we denote by $S(X \ast H)$ the vector space of all its Borel sections. If $\mu$ is a measure on $X$, then we may construct a Hilbert space associated to the Borel Hilbert bundle, called the \emph{direct integral}, given as follows: The space $$ \mathcal{H} := \left\lbrace  f \in S(X \ast H) \mid x \mapsto \|f(x)\|_{H(x)}^{2} \text{ is $\mu$-integrable} \right\rbrace  ,$$ is a pre-inner product space with the pre-inner product given by: $$ \langle f , g \rangle = \int_{X} \langle f (x) , g(x) \rangle_{H(x)} \, d \mu(x) .$$ Upon taking the quotient of norm-zero elements, we obtain the direct integral, denoted by $\int_{X}^{\oplus} H(x) \, d \mu(x)$.

The following proposition will be used frequently throughout the document. By a bounded Borel section of a Borel Hilbert bundle $X \ast H$, we mean a Borel section $\xi \colon X \to X \ast H$ such that the map $x \mapsto \|\xi(x)\|_{H(x)}$ is bounded.
\begin{lemma} \label{lem: dense subspace given by fundamental sequence} \cite[p.172]{Dixmier:VonNeumannAlgebras}
	Let $X \ast H$ be a Borel Hilbert bundle over a standard Borel space X with fundamental sequence $\{f_n\}_n$, and let $\mu$ denote a measure on X. Suppose that each Borel section $f_n$ in the fundamental sequence is bounded. Let $L_{\mathrm{fin}}^{\infty}(X, \mu)$ denote the set of bounded Borel functions which vanish off a set of finite measure. Given $a \in L_{\mathrm{fin}}^{\infty}(X, \mu)$ and $f$ a Borel section, denote by $T_a (f)$ the Borel section $T_a (f)(x) = a(x) f(x)$. Then the linear span of the set $$\{ T_a (f_n) \mid a \in L_{\mathrm{fin}}^{\infty}(X, \mu) ; n \in \mathbb{N}\},$$ is dense in the direct integral $\int_{X}^{\oplus} H(x) \, d\mu(x)$.
\end{lemma}

Suppose that for each $i \in \N$, we have a Borel Hilbert bundle $X \ast H_i$. For each $x \in X$, let $H(x)$ denote the Hilbert space direct sum $H(x) := \bigoplus_{i = 1}^{\infty} H_i (x)$. It follows by \cref{prop: family of sections defining standard Borel structure on the Hilbert bundle PaperA} that $X \ast H$ is in a natural way another Borel Hilbert bundle with the following fundamental sequence: If $\{f_{n,m}\}_{n}$ is a fundamental sequence of the Hilbert bundle $X \ast H_m$, for $m \in \N$, then identifying these as sections in $X \ast H$ yields a fundamental sequence $\{ f_{n,m} \}_{m,n = 1}^{\infty}$ for $X \ast H$. The proof of the next result is straightforward and so we omit it.

\begin{lemma} \label{lem: direct integrals commutes with direct sums} 
	With the setup as in the previous paragraph, the map $$ U \colon \bigoplus_{i = 1}^{\infty} \int_{X}^{\oplus} H_i (x) \, d\mu (x) \to \int_{X}^{\oplus} \bigoplus_{i = 1}^{\infty} H_i (x) \, d \mu(x) ,$$ given by $ U(\xi_i)_i (x) = (\xi_i (x))_{i \in \N} \in H(x) $, is an isomorphism of Hilbert spaces. 
\end{lemma}

\begin{definition} \label{def: isomorphism groupoid PaperA}
	Let $X \ast H$ be a Borel Hilbert bundle over a standard Borel space $X$, and let $\{f_n \}_n$ be a fundamental sequence for $X \ast H$. The \emph{isomorphism groupoid} associated to the bundle $X \ast H$ is the groupoid $$ \text{Iso}(X \ast H) := \left\lbrace  (x,T,y) \mid x,y \in X \text{ and } T \in \mathcal{U}(H(y), H(x)) \right\rbrace ,$$ where, for $x,y \in X$, $\mathcal{U}(H(y),H(x))$ denotes the set of Hilbert space isomorphisms $T \colon H(y) \to H(x)$. With the canonical operations $(x,T,y) (y,R,z) = (x,TR,z)$ and $(x,T,y)^{-1} = (y,T^{-1},x)$, $\text{Iso}(X \ast H)$ becomes a groupoid. Moreover, endowing $\text{Iso}(X \ast H)$ with the weakest Borel structure for which the maps $$ (x,T,y) \mapsto \langle T f_n(y) , f_m(x) \rangle_{H(x)} $$ are Borel for all $n,m \in \N$, it becomes a standard Borel space and the groupoid operations are Borel.
\end{definition}

\subsection{Étale groupoids} \label{subsec: etale gpds}
A \emph{groupoid} is a set $\G$ equipped with a partially defined multiplication (composition) $\G^{(2)} \to \G \, , \, (x,y) \mapsto xy$, where $\G^{(2)} \subset \G \times \G$ is the set of composable pairs, and with an inverse map $\G \to \G \ , \, x \mapsto x^{-1}$, such that the following three axioms are satisfied:

\begin{itemize}
	\item If $(x , y), (y, z) \in \G^{(2)}$, then $(x y , z) , (x , yz) \in \G^{(2)}$ and $(xy ) z = x (y z)$;
	\item $(x^{-1})^{-1} = x$, for all $x \in \G$;
	\item $(x, x^{-1}) \in \G^{(2)}$, for all $x \in \G$, and when $(x,y) \in G^{(2)}$, we have $x^{-1} (xy) = y$ and $(xy)y^{-1} = x$.
\end{itemize}

The set $\G^{(0)} := \{ x x^{-1} \colon x \in \G \}$ is called the \emph{unit space} of the groupoid $\G,$ and the maps $r \colon \G \to \G , r(x) = x x^{-1}$ and $s \colon \G \to \G , x \mapsto x^{-1} x$ are called the \emph{range} and \emph{source} maps respectively. We have that $(x,y) \in \G^{(2)}$ if and only if $s(x) = r(y)$. A groupoid endowed with a locally compact topology such that multiplication and inversion are continuous is called a \emph{topological groupoid}. Moreover, if the topology is such that the range map, and therefore also the source map, is a local homeomorphism, the groupoid is said to be \emph{étale}. The sets $U \subset \G$ for which $s |_U$ and $r |_U$ are injective are called \emph{bisections}. Thus, an étale groupoid is one whose topology has a basis of open bisections. For any $X \subset \G^{(0)} ,$ we denote by $\G_X = \{x \in \G \colon s(x) \in X\} $ and $\G^{X} = \{x \in \G \colon r(x) \in X\}$. We shall write $\G_u$ and $\G^{u}$ instead of $\G_{\{u\}}$ and $G^{\{u\}}$ whenever $u \in \G^{(0)}$ is a unit. The set $$\G (X) = \G_X \cap \G^{X} = \{x \in \G \colon r(x) , s(x) \in X\} ,$$ is a subgroupoid of $\G$, with unit space $X$, called the \emph{restriction} of $\G$ to $X$. A subset $X \subset \G^{(0)}$ is called \emph{invariant} if for all $g \in \G$, we have that $s(g) \in X$ if and only if $r(g) \in X$.
The \emph{isotropy group} at a unit $u \in \G^{(0)}$ is the group $\G(u) = \G_u \cap \G^{u}$ and the \emph{isotropy bundle} is $$\text{Iso}(\G) := \{ x \in \G \colon s(x) = r(x) \} = \bigsqcup_{u \in \G^{(0)}} \G(u) .$$ A \emph{homomorphism} between two étale groupoids $\G$ and $\mathcal{H}$, is a map $\phi \colon \G \to \mathcal{H}$ such that if $(x,y) \in \G^{(2)}$, then $(\phi(x) , \phi(y)) \in \mathcal{H}^{(2)}$, and in this case $\phi(xy) = \phi(x) \phi(y)$. Two étale groupoids are said to be \emph{isomorphic} if there is a bijective groupoid homomorphism that is also a homeomorphism. A \emph{Borel homomorphism} of groupoids is a homomorphism of groupoids which is also a Borel map. 

One associates to a Hausdorff étale groupoid $\G$ an associative and involutive normed algebra in the following manner: Since $\G$ is étale, the fibers $\G_u$ and $\G^u$, for $u \in \G^{(0)}$, are discrete. Let $C_c (\G)$ denote the space of continuous compactly supported functions on $\G$. We endow $C_c (\G)$ with the convolution product, which for $f,g \in C_c (\G)$ is given by $$ f \ast g (x) = \sum_{y \in \G_{s(x)}} f(x y^{-1}) g(y) = \sum_{y \in \G^{r(x)}} f(y) g(y^{-1}x) ,$$ for $x \in \G$. The involution is defined by $$ f^{\ast}(x) = \overline{f(x^{-1})} ,$$ for $f \in C_c (\G)$ and $x \in \G$, and the \emph{$I$-norm} on $C_c (\G)$ is given by $$ \|f\|_I = \max \left\lbrace \sup_{u \in \G^{(0)}} \sum_{x \in \G_u} |f(x)| \, , \, \sup_{u \in \G^{(0)}} \sum_{x \in \G^{u}} |f(x)| \right\rbrace .$$ With the above norm and algebraic operations, $(C_c (\G), \ast, ^{\ast}, \|\cdot\|_{I})$ becomes an involutive normed algebra.

The commutative algebra of bounded Borel functions on $\G$ will be denoted by $B(\G)$, and the ideal in $B(\G)$ of bounded compactly supported Borel functions by $B_c (\G)$.

\subsection{Unitary representations of étale groupoids} \label{subsec: unitary repr. of gpds}
From this point going forward, we assume that $\G$ is a second-countable Hausdorff étale groupoid. In particular, $\G$ and $\G^{(0)}$ are both standard Borel spaces.

\begin{definition} \label{def: unitary representation of groupoids PaperA}
	A \emph{unitary representation} of an étale groupoid $\G$ is a Borel homomorphism $\pi \colon \G \to \text{Iso}(\G^{(0)} \ast H_\pi)$ which preserves the unit space, in the sense that for all $x \in \G$, $\pi(x) = (r(x) , \hat{\pi}(x) , s(x))$, where $\hat{\pi}(x) \in \mathcal{U}(H_\pi (s(x)) , H_\pi (r(x)))$.
\end{definition}
We shall usually identify a given unitary representation $\pi$ with $\hat{\pi}$, so that when we write $\pi(x)$ we shall often mean $\hat{\pi}(x) \in \mathcal{U}(H_\pi (s(x)) , H_\pi (r(x)))$.

\begin{definition} \label{def: quasi-invariant measure}
	Let $\mu$ be a Radon measure on $\G^{(0)}$. The map $$ f \mapsto \int_{\G^{(0)}} \sum_{x \in \G^{u}} f(x) \, d \mu (u) $$ is a positive linear functional on $C_c (\G)$, and hence by the Riesz-Markov-Kakutani theorem there exists a Radon measure $\nu_\mu$ on $\G$ such that $$ \int_{\G^{(0)}} \sum_{x \in \G^{u}} f(x) \, d \mu (u) = \nu_\mu (f) ,$$ for all $f \in C_c (\G)$. Let $\nu_{\mu}^{-1}$ denote the measure given by $\nu_{\mu}^{-1}(E) = \nu_\mu (E^{-1})$, for every Borel set $E$. Then we have that $$ \int_{\G^{(0)}} \sum_{x \in \G_{u}} f(x) \, d \mu (u) = \nu_{\mu}^{-1}(f) ,$$ for all $f \in C_c (\G)$. The Radon measure $\mu$ is said to be \emph{invariant} if $\nu_\mu = \nu_{\mu}^{-1}$, and \emph{quasi-invariant} if $\nu_\mu$ and $\nu_{\mu}^{-1}$ are mutually absolutely continuous. 
\end{definition}
We shall usually denote the induced measure of $\mu$ by $\nu$ instead of $\nu_\mu$, when there is no danger of confusion. 
When $\G$ is an étale groupoid, a Radon measure $\mu$ on $\G^{(0)}$ is quasi-invariant if and only if the following holds: for any Borel bisection $E \subset \G$, we have that $\mu(r(E)) = 0$ if and only if $\mu(s(E)) = 0$ (see \cite[Proposition 3.2.2]{Paterson:GroupoidsAndTheirOperatorAlgebras}). For a family $\mathcal{M}$ consisting of quasi-invariant measures, we define its \emph{support} to be the set $\bigcup_{\mu \in \mathcal{M}} \supp (\mu) $.
When $\mu$ is quasi-invariant, the Radon-Nikodym theorem gives a Borel map $\Delta = \frac{d \nu}{d \nu^{-1}}$ on $\G$ called the \emph{modular function} for $\mu$. One may in fact choose $\Delta$ so that it becomes a Borel homomorphism from $\G$ to the multiplicative group of positive real numbers (see \cite{Hahn:HaarMeasureForMeasuredGroupoids,Ramsay:TopologiesOnMeasuredGroupoids}), and we will always do this; in particular, $\Delta(xy) = \Delta (x) \Delta(y)$, for all $(x,y) \in \G^{(2)}$, and $\Delta(x^{-1}) = \Delta(x)^{-1}$, for all $x \in \G$. From a quasi-invariant measure $\mu$, we may also define a Radon measure $\nu_0$ by $d \nu_0 := \Delta^{-1/2} \, d \nu = \Delta^{1/2} \, d \nu^{-1}$. The measure $\nu_0$ satisfies that $\nu_0 (E) = \nu_0 (E^{-1})$, for all Borel sets $E \subset \G$.

\begin{theorem} \cite[Theorem 1.21 and Proposition 1.7]{Renault:AGroupoidApproach} \label{thm: fundamental theorem of analysis on groupoids PaperA}
	Let $\pi$ be a unitary representation of $\G$, with associated Borel Hilbert bundle $\G^{(0)} \ast H_\pi$. If $\mu$ is a quasi-invariant measure on $\G^{(0)}$, then $\pi$ integrates to an $I$-norm bounded representation of $C_c (\G)$ on $\int_{\G^{(0)}}^{\oplus} H_\pi (x) \, d \mu(x)$, denoted $\pi_\mu$, such that $$ \langle \pi_\mu (f) \xi , \eta \rangle = \int_{\G} f(x) \langle \pi(x) \xi(s(x)) , \eta(r(x)) \rangle_{H_\pi(r(x))} \, d \nu_0 (x) ,$$ for $\xi , \eta \in \int_{\G^{(0)}}^{\oplus} H_\pi(x) \, d \mu(x)$ . Explicitly, $\pi_\mu$ is given by $$ \pi_\mu (f)\xi (u) = \sum_{x \in \G^{u}} f (x) \pi(x)(\xi (s(x))) \Delta(x)^{-1/2} ,$$ for $\xi \in \int_{\G^{(0)}}^{\oplus} H_\pi(u) \, d \mu(u)$ and $u \in \G^{(0)}$. 
	
	Conversely, every representation of $C_c (\G)$ on a separable Hilbert space is unitarily equivalent to $\pi_\mu$ for some unitary representation $\pi$ on $\G$ and quasi-invariant measure $\mu$.
\end{theorem}

The representation $\pi_\mu$ of $C_c (\G)$ is called the \emph{integrated form} of $\pi$ with respect to the quasi-invariant measure $\mu$. Using well-known identities associated with the Radon-Nikodym derivative, it is not hard to see that if $\mu \sim \mu^\prime$ are equivalent quasi-invariant measures, then $\pi_\mu \sim \pi_{\mu^\prime}$ are unitarily equivalent representations. We shall sometimes denote by $H_{\pi,\mu}$ the direct integral $\int_{\G^{(0)}}^{\oplus} H_\pi(u) \, d \mu(u)$. 

\begin{example} \label{ex: left regular representation induced from unitary groupoids representation}
	Let $\G$ be a second-countable étale groupoid. The \emph{left regular representation} is the unitary representation $\lambda$, where $H_\lambda (u) = \ell^{2}(\G^u)$, for each $u \in \G^{(0)}$, and $\lambda \colon \G \to \text{Iso}(\G^{(0)} \ast H_\lambda)$ is given by $$\lambda(x) \colon \ell^{2}(\G^{s(x)}) \to \ell^2 (\G^{r(x)}) \, , \, \lambda(x) \xi (y) = \xi (x^{-1} y).$$ A fundamental sequence for the Borel Hilbert bundle $\G^{(0)} \ast H_\lambda$ is given by taking any sequence $\{f_n\}_n$ that is dense in the inductive limit topology of $C_c (\G)$, and identifying each such $f_n$ with the section $u \mapsto {f_n}_{|_{\G^{u}}} \in \ell^{2}(\G^{u})$. 
	Fix any quasi-invariant measure $\mu$ on $\G^{(0)}$. The direct integral $\int_{\G^{(0)}}^{\oplus} \ell^2 (\G^u) \, d \mu(u)$ is canonically identified with the Hilbert space $L^2 (\G, \nu)$, by the identification of an $L^2 (\G, \nu)$ function $\xi$ with the section given by $u \mapsto \xi_{|_{\G^{u}}}$. Under this identification, the integrated form of the left regular representation with respect to $\mu$ is given by $$ \lambda_\mu (f)(\xi)(y) = \sum_{x \in \G^{r(y)}} f(x) \lambda(x) (\xi)(y) \Delta(x)^{-1/2} = \sum_{x \in \G^{r(y)}} f(x) \xi(x^{-1}y) \Delta(x)^{-1/2} .$$ 
	Pointwise multiplication with the Borel homomorphism $\Delta^{-1/2}$ is a Hilbert space isomorphism from $L^{2}(\G, \nu^{-1})$ to $L^2 (\G , \nu)$, and under this identification the above shows that $\lambda_\mu$ is just given by the usual left convolution on $L^2 (\G , \nu^{-1})$ (without the $\Delta$-term). Whenever we write $\text{Ind}(\mu)$, we shall mean the representation of $C_c (\G)$ sending a function to the corresponding operator on $L^2 (\G , \nu^{-1})$ given by left convolution. Thus $\text{Ind}(\mu)$ is unitarily equivalent to $\lambda_\mu$, for any quasi-invariant measure $\mu$. In fact, one can check that $\text{Ind}(\mu)$ defines a representation of $C_c (\G)$ for any Radon measure $\mu$, not only the quasi-invariant ones.
\end{example}

\begin{definition} \label{def: the full and reduced C*-algebras of a groupoid in terms of its unitary representations}
	The \emph{full $\rm C^*$-algebra} associated to the groupoid $\G$ is $\Cfull(\G) := \overline{C_c (\G)}^{\| \cdot \|_{max}} ,$ where for any $f \in C_c (\G)$, $$ \|f\|_{max} := \sup_{\mu , \pi} \|\pi_\mu (f)\| ,$$ where the supremum is taken over all unitary representations $\pi$ of the groupoid and all quasi-invariant measures $\mu$ on $\G^{(0)}$. The \emph{reduced $\rm C^*$-algebra} associated to $\G$ is $C_{r}^{\ast}(\G) := \overline{C_c (\G)}^{\|\cdot\|_{r}},$ where for any $f \in C_c (\G)$, $$\|f\|_{r} := \sup_{\mu} \|\lambda_\mu (f)\| = \sup_{\mu} \|\text{Ind}(\mu)(f)\| ,$$ where the supremum is taken over all quasi-invariant measures $\mu$ on $\G^{(0)}$. The Dirac measures $\delta_u$, for $u \in \G^{(0)}$, are probability measures on $\G^{(0)}$, and for any Radon measure $\mu$, we have $\|\text{Ind}(\mu) (f)\| = \sup_{u \in \supp(\mu)} \|\text{Ind}(\delta_u) (f)\|$, for any $f \in C_c (\G)$ (see \cite[Proposition 3.1.2]{Paterson:GroupoidsAndTheirOperatorAlgebras}). The transitive measures (see \cite[Definition 3.9]{Renault:AGroupoidApproach}) on an étale groupoid $\G$ are quasi-invariant, and the combined support of these covers the unit space. Using this together with the above, we arrive at the more well-known description of the reduced norm for an element $f \in C_c (\G)$, namely $$ \|f\|_{r} = \sup_{u \in \G^{(0)}} \|\text{Ind}(\delta_u)(f)\| .$$ 
\end{definition}

\section{Constructing groupoid $\rm C^*$-algebras} \label{sec: constructing grouopid c*-algebras}
We now turn to the task of associating to a groupoid $\G$, $\rm C^*$-algebras from given pairs $(D, \mathcal{M})$, where $D$ is an algebraic ideal inside $B(\G)$, the algebra of bounded complex-valued Borel functions on $\G$, and $\mathcal{M}$ is a non-empty family of quasi-invariant measures on the unit space.

\begin{definition} \label{def: D-representations PaperA}
	Let $D \trianglelefteq B(\G)$ be an algebraic ideal. A unitary representation $\pi$ of $\G$ is a \emph{$D$-representation} if there exists a fundamental sequence $\{f_n\}_n$ of the Borel Hilbert bundle $\G^{(0)} \ast H_\pi$, such that for all $n, m \in \N$, we have that the Borel function $$ x \mapsto \langle \pi(x) f_n(s(x)), f_m (r(x)) \rangle_{H_\pi (r(x))} ,$$ is an element of $D$. 
\end{definition}

\begin{remark} \label{rmk: how the groupoid defintion of unitary D-reps generalize the group one}
	Suppose that $\G$ is a group and that $D \trianglelefteq \ell^{\infty} (\G)$ is an ideal. Then a Borel Hilbert bundle over the identity element consists of a single separable Hilbert space $H$, and a fundamental sequence is just a sequence of vectors densely spanning $H$. Suppose that $\pi$ is a unitary representation which is a $D$-representation in the sense of \cref{def: D-representations PaperA} for some spanning sequence. Then given any pair of vectors in the dense subspace given by the linear span of this sequence, the corresponding matrix coefficients are clearly also in $D$. Conversely, if $\pi$ is a $D$-representation in the sense of \cite[Definition 2.1]{BrownandGuentner:NewC*-completions}, then by definition there exists a dense linear subspace $H_0$ of $H_\pi$ such that the matrix coefficients corresponding to vectors in $H_0$ are in $D$. Taking any spanning sequence of $H_\pi$ in $H_0$, we see that $\pi$ is a $D$-representation in the sense of \cref{def: D-representations PaperA}. Therefore, \cref{def: D-representations PaperA} reduces to that of \cite[Definition 2.1]{BrownandGuentner:NewC*-completions} when $\G$ is a group.
\end{remark}

\begin{lemma} \label{lema: matrix coefficients in D for fundamental sequence imply the same for the dense subspace they span}
	Let $\pi$ be a unitary representation of $\G$ and let $D \trianglelefteq B(\G)$ be an algebraic ideal. Fix a quasi-invariant measure $\mu$ on the unit space $\G^{(0)}$. If $\pi$ is a $D$-representation, then there exists a dense subspace $H_0 \subset H_{\pi, \mu}$ such that for all $\xi , \eta \in H_0$, the map $x \mapsto \langle \pi(x) \xi(s(x)) , \eta (r(x)) \rangle_{H_\pi (r(x))}$ is in $D$.
	\begin{proof}
		By definition, there exists a fundamental sequence $\{f_n\}_n$ for the Borel Hilbert bundle $\G^{(0)} \ast H_\pi$ such that for every $m,n \in \N$, the map $x \mapsto \langle \pi(x) f_n (s(x)) , f_m (r(x)) \rangle_{H_\pi (r(x))}$ is in $D$. In particular, for any $n \in \N$, the function $$u \mapsto \langle \pi(u) f_n (u) , f_n (u) \rangle_{H_\pi (u)} = \langle f_n (u) , f_n (u) \rangle_{H_\pi (u)} = \| f_n(u) \|_{H_\pi (u)}^{2}$$ is bounded on the unit space. By \cite[Proposition 7, p.172]{Dixmier:VonNeumannAlgebras}, the linear span of the set in \cref{lem: dense subspace given by fundamental sequence} is dense in the direct integral $H_{\pi, \mu}$. Given any vectors $\xi$ and $\eta$ in said linear span, there are finitely many functions $a_i , b_j \in L_{\mathrm{fin}}^{\infty} (\G^{(0)} , \mu)$ such that $\xi = \sum_i T_{a_i} (f_i)$ and $\eta = \sum_j T_{b_j} (f_j)$. Since $D$ is an ideal, it follows that the function
		\begin{align*}
			x \mapsto \langle \pi(x) \xi(s(x)), \eta(r(x)) \rangle_{H_\pi (r(x))} &= \left\langle \sum_i a_i (s(x)) \pi(x) f_i (s(x)) , \sum_j b_j (r(x)) f_j (r(x)) \right\rangle_{H_\pi (r(x))} \\
			&= \sum_{i,j} a_i (s(x)) \overline{b_j (r(x))} \langle \pi(x) f_i (s(x)) , f_j (r(x)) \rangle_{H_{\pi} (r(x))}
		\end{align*}
		is in $D$ as well.
	\end{proof}
\end{lemma}

\begin{definition} \label{def: C*-algebra associated with an Ideal + family of measures PaperA}
	Let $D \trianglelefteq B(\G)$ be an algebraic ideal and $\mathcal{M}$ a non-empty family of quasi-invariant measures on $\G^{(0)}$. The associated $\rm C^*$-algebra is given as follows: Define a $\rm C^*$-seminorm on $C_c (\G)$ by $$\|f\|_{D, \mathcal{M}} := \sup\left\lbrace \|\pi_\mu (f)\| \colon \, \pi \text{ is a } D\text{-representation and } \mu \in \mathcal{M} \right\rbrace.$$ Putting $\mathcal{N}_{D , \mathcal{M}} := \left\lbrace f \in C_c (\G) \colon \|f\|_{D , \mathcal{M}} = 0 \right\rbrace$, we define $$C_{D , \mathcal{M}}^{\ast} (\G) := \overline{C_c (\G)/\mathcal{N}_{D , \mathcal{M}}}^{\|\cdot\|_{D , \mathcal{M}}} .$$
\end{definition}

\begin{remark} \label{rmk: ideal + measure completion generalize Brown and Guentners definition}
	When $\G$ is a countable discrete group, there is of course only one probability measure on the unit space (which is invariant) and as mentioned in \cref{rmk: how the groupoid defintion of unitary D-reps generalize the group one}, \cite[Definition 2.1]{BrownandGuentner:NewC*-completions} and \cref{def: D-representations PaperA} are equivalent in this case. Therefore, \cref{def: C*-algebra associated with an Ideal + family of measures PaperA} reduces to \cite[Definition 2.2]{BrownandGuentner:NewC*-completions} for countable discrete groups.
\end{remark}

We shall write $C_{D, \mu}^{\ast} (\G)$ instead of $C_{D, \{\mu\}}^{\ast} (\G)$. When $\mathcal{M}$ denotes the family of all quasi-invariant measures, we shall write  $C_{D}^{\ast} (\G)$ instead of $C_{D, \mathcal{M}}^{\ast} (\G)$.

\begin{proposition} \label{prop: full C*-algebra is the B(G)-repr. completion}
	$C_{B(\G)}^{*}(\G) = \Cfull(\G)$.
	\begin{proof}
		This follows since any unitary representation is a $B(\G)$-representation. Indeed, let $\pi$ be a unitary representation with associated Borel Hilbert bundle $\G^{(0)} \ast H_\pi$, and let $\{ f_n \}_n$ be an orthonormal fundamental sequence for $\G^{(0)} \ast H_\pi$. Then for any $m,n \in \N$, by Cauchy-Schwartz, 
		\begin{align*}
			|\langle \pi(x) f_n (s(x)), f_m (r(x)) \rangle_{H_\pi(r(x))}| &\leq \|\pi(x) f_n (s(x)) \|_{H_\pi(r(x))}\| f_m (r(x)) \|_{H_\pi(r(x))} \\
			&= \| f_n (s(x)) \|_{H_\pi(s(x))} \| f_m (r(x)) \|_{H_\pi(r(x))} \leq 1,
		\end{align*}
		for $x \in \G$, and we already know that the map $x \mapsto \langle \pi(x) f_n (s(x)), f_m (r(x)) \rangle_{H_\pi(r(x))}$ is Borel.
	\end{proof}
\end{proposition}

Recall that $B_c (\G)$ denotes the set of compactly supported bounded Borel functions. 

\begin{lemma} \label{lem: left regular representation is a B_c (G)-representation}
	The left regular representation is a $B_c (\G)$-representation.
	\begin{proof}
		Let $\lambda$ denote the left regular representation on $\G$. We take as our fundamental sequence a dense sequence in the inductive limit topology of $C_c (\G)$ identified in the canonical way as sections. For two such sections $f, g \in C_c (\G)$, we have
		\begin{align*}
			\langle \lambda(x) f(s(x)), g(r(x)) \rangle_{\ell^{2}(\G^{r(x)})} &= \sum_{y \in \G^{r(x)}} \lambda(x) f(s(x))(y) \overline{g(r(x))(y)} \\
			&= \sum_{y \in \G^{r(x)}} f(x^{-1} y) \overline{g(y)} \\
			&= \sum_{y \in \G^{r(x)}} \overline{g(y)} \, \overline{f^{\ast} (y^{-1}x)} = \overline{g} \ast \overline{f^\ast} (x),
		\end{align*}
		and clearly the function $$x \mapsto \langle \lambda(x) f(s(x)), g(r(x)) \rangle_{\ell^{2}(\G^{r(x)})} = \overline{g} \ast \overline{f^\ast} (x) ,$$ is in $C_c (\G) \subset B_c (\G)$.
	\end{proof}
\end{lemma}

\begin{remark} \label{rmk: some remarks regarding ideal completions}
	From \cref{def: C*-algebra associated with an Ideal + family of measures PaperA}, a priori, it may well be that for some algebraic ideals $D \trianglelefteq B(\G)$ and family of measures $\mathcal{M}$, the ideal $\mathcal{N}_{D, \mathcal{M}}$ is non-trivial, in which case $C_{D, \mathcal{M}}^{\ast} (\G)$ is strictly speaking not a completion of $C_c (\G)$. However, by \cref{lem: left regular representation is a B_c (G)-representation} and the proof of \cite[Proposition 1.11]{Renault:AGroupoidApproach}, if $C_c (\G) \subset D$ and $\mathcal{M}$ has full support, then $C_{D, \mathcal{M}}^{\ast}(\G)$ is indeed a completion of $C_c (\G)$; because in this case, given any $f \in C_c (\G)$, we have that $$ \|f\|_{r} = \sup_{u \in \G^{(0)}} \|\text{Ind}(\delta_u) (f)\| = \sup_{\mu \in \mathcal{M}} \|\mathrm{Ind}(\mu) (f)\| \leq \|f\|_{D, \mathcal{M}} ,$$ since $\mathcal{M}$ has full support and $\text{Ind}(\mu)$, for $\mu \in \mathcal{M}$, are (unitarily equivalent to) $D$-representations.
\end{remark}

Let us look at some basic properties regarding this construction.

\begin{lemma} \label{lem: ideal containment induces quotient map of c*-algebras}
	Let $D_1 , D_2 \trianglelefteq B(\G)$ be two algebraic ideals and $\mathcal{M}_1 , \mathcal{M}_2$ two families of quasi-invariant measures. If $D_1 \subset D_2$ and $\mathcal{M}_1 \subset \mathcal{M}_2$, then there is a canonical surjection $C_{D_2 , \mathcal{M}_2}^{\ast}(\G) \to C_{D_1 , \mathcal{M}_1}^{\ast}(\G)$.
	\begin{proof}
		If $D_1 \subset D_2$ and $\mathcal{M}_1 \subset \mathcal{M}_2$, then $\|\cdot\|_{D_2 , \mathcal{M}_2} \geq \|\cdot\|_{D_1 , \mathcal{M}_1}$ and $\mathcal{N}_{D_2 , \mathcal{M}_2} \subset \mathcal{N}_{D_1 , \mathcal{M}_1}$. Therefore, the quotient map $ C_c (\G)/\mathcal{N}_{D_2 , \mathcal{M}_2} \to C_c (\G)/\mathcal{N}_{D_1 , \mathcal{M}_1} ,$ extends to a surjective $*$-homomorphism $C_{D_2 , \mathcal{M}_2}^{\ast}(\G) \to C_{D_1 , \mathcal{M}_1}^{\ast} (\G)$.
	\end{proof}
\end{lemma}

\begin{lemma} \label{lem: tensor product of a D-representation with another is a D-representation}
	Let $\G$ be a groupoid and fix $D \trianglelefteq B(\G)$ an algebraic ideal. Suppose that $\rho$ is a $D$-representation and that $\pi$ is a unitary representation. Then the tensor product $\rho \otimes \pi$ is a $D$-representation.
	\begin{proof}
		This follows from the fact that $D$ is an ideal. Indeed, by \cite[Proposition 10 p. 174]{Dixmier:VonNeumannAlgebras}, if $\{ \xi_i \}_i$ and $\{ \eta_j \}_j$ are fundamental sequences, respectively for $\G^{(0)} \ast H_\pi$ and $\G^{(0)} \ast H_\rho$, then $\{ \xi_i \otimes \eta_j \}_{i,j}$ form a fundamental sequence for the Hilbert bundle $\G^{(0)} \ast (H_\pi \otimes H_\rho)$. Here $H_\pi \otimes H_\rho (u) = H_\pi(u) \otimes H_\rho(u)$, for $u \in \G^{(0)}$. Assuming the fundamental sequence $\{ \xi_i \}_i$ is such that the map $x \mapsto \langle \rho(x) \xi_n(s(x)), \xi_m (r(x)) \rangle_{H_\rho (r(x))}$ is in $D$ for all $n,m \in \N$, and $\{\eta_j \}_j$ is any fundamental sequence for $\G^{(0)} \ast H_\pi$, then since $D$ is an ideal, we have that
		\begin{gather*}
			x \mapsto \langle \rho \otimes \pi (x) (\xi_n \otimes \eta_i)(s(x)) , \xi_m \otimes \eta_j (r(x)) \rangle_{H_\rho (r(x)) \otimes H_\pi (r(x))} \\
			= \langle \rho(x) \xi_n (s(x)), \xi_m (r(x)) \rangle_{H_\rho (r(x))} \langle \pi(x) \eta_i (s(x)), \eta_j (r(x)) \rangle_{H_\pi (r(x))} ,
		\end{gather*}
		is in $D$ also.
	\end{proof}
\end{lemma}

\begin{definition} \label{def: direct sum of countably many unitary representations}
	Assume that $\{ \rho_i \}_{i \in \N}$ is a collection of unitary representations. For each unit $u \in \G^{(0)}$, let $H(u)$ be the Hilbert space direct sum $H(u) := \bigoplus_{i = 1}^{\infty} H_{\rho_i} (u)$ and endow the Hilbert bundle $\G^{(0)} \ast H$ with the canonical standard Borel structure and fundamental sequence. The unitary representation $\rho := \oplus_{i = 1}^{\infty} \rho_i $ is called the \emph{direct sum representation}.
\end{definition}

\begin{proposition} \label{prop: direct sum of D-repr. is again a D-repr.}
	Let $D \trianglelefteq B(\G)$ be an algebraic ideal, and assume that $\{ \rho_i \}_{i \in \N}$ is a sequence of $D$-representations. Then the direct sum representation $\oplus_{i = 1}^{\infty} \rho_i$ is a $D$-representation. Moreover, if $\mu$ is a quasi-invariant measure, then the integrated form of $\oplus_{i = 1}^{\infty} \rho_i$ with respect to $\mu$ is unitarily equivalent to the direct sum of the representations $(\rho_{i})_{\mu}$.
	\begin{proof}
		For the first statement, recall that if $\{f_{n,m}\}_{n}$ is a fundamental sequence of the Hilbert bundle $\G^{(0)} \ast H_{\rho_m}$, for $m \in \N$, then identifying these as sections in $\G^{(0)} \ast H$, where $H(u) = \bigoplus_{m = 1}^{\infty} H_{\rho_m} (u)$, yields the fundamental sequence $\{ f_{n,m} \}_{m,n = 1}^{\infty}$ for $\G^{(0)} \ast H$. Thus, if all $\rho_m$ are $D$-representations, then clearly the function $$\langle \oplus_m \rho_m (x) f_{i,j} (s(x)) , f_{k,l}(r(x)) \rangle_{H(r(x))} = \begin{cases}
			\langle \rho_j (x) f_{i,j}(s(x)) , f_{k,j}(r(x)) \rangle_{H_{\rho_j} (r(x))} & \text{ if } $j = l$, \\
			0 & \text{ else },
		\end{cases} $$ is in $D$.
		
		For the second statement, we shall show that the Hilbert space isomorphism $U$ from \cref{lem: direct integrals commutes with direct sums} implements a unitary equivalence between the two representations. For ease of notation, denote these integrated forms by $\Pi_{\oplus_{i = 1}^{\infty} \rho_i}$ and $\oplus_{i = 1}^{\infty} \Pi_{\rho_i}$. Recall that the Hilbert space isomorphism $U$ was given as follows: Given $(\xi_i)_i$, $U(\xi_i)_i$ is defined to be the (Borel) section of $\G^{(0)} \ast H$ such that $U(\xi_i) (u) = (\xi_i (u))_{i \in \N} \in H(u)$. 
		We shall show that for any $f \in C_c (\G)$, we have that $\Pi_{\oplus_i \rho_i} (f) U (\xi) = U \oplus_{i = 1}^{\infty} \Pi_{\rho_i} (f) (\xi)$ for all $\xi \in \bigoplus_{i = 1}^{\infty} \int_{\G^{(0)}}^{\oplus} H_i (u) \, d \mu_i (u)$. It suffices to prove this for $\xi$ in the algebraic direct sum, as these elements are dense. Moreover, it suffices to show that $\langle \Pi_{\oplus_i \rho_i} (f) U (\xi) , U (\eta) \rangle = \langle \oplus_{i = 1}^{\infty} \Pi_{\rho_i} (f) (\xi), \eta \rangle $, for all $\xi, \eta$ in the algebraic direct sum. Fix such $f, \xi$ and $\eta$. Then 
		\begin{align*}
			\langle \Pi_{\oplus_i \rho_i} (f) U \xi , U \eta \rangle &= \int_{\G^{(0)}} \sum_{x \in \G^{u}} f(x) \Delta^{-1/2}(x) \langle \oplus_i \rho_i (x) U \xi(s(x)), U \eta(r(x)) \rangle_{H(r(x))} \, d \mu(u) \\
			&= \int_{\G^{(0)}} \sum_{x \in \G^{u}} \sum_{i = 1}^{\infty} f(x) \Delta^{-1/2}(x) \langle \rho_i (x) \xi_i (s(x)) , \eta_i(r(x)) \rangle_{H_{\rho_i} (r(x))} \, d \mu(u) \\
			&= \sum_{i = 1}^{\infty} \int_{\G^{(0)}} \sum_{x \in \G^{u}} f(x) \Delta^{-1/2}(x) \langle \rho_i (x) \xi_i (s(x)) , \eta_i(r(x)) \rangle_{H_{\rho_i} (r(x))} \, d \mu (u)  \\
			&= \sum_{i = 1}^{\infty} \langle \Pi_{\rho_i}(f) \xi_i , \eta_i \rangle = \langle \oplus_{i = 1}^{\infty} \Pi_{\rho_i} (f) \xi, \eta \rangle.
		\end{align*}
		This proves the proposition.
	\end{proof}
\end{proposition}

Given an ideal $D \trianglelefteq B(\G)$ and a family of quasi-invariant measures $\mathcal{M}$, we shall say that a $D$-representation $\pi$ is \emph{faithful with respect to $\mathcal{M}$} if $\| a \|_{D, \mathcal{M}} = \sup_{\mu \in \mathcal{M}} \| \pi_\mu (a) \|$, for all $a \in C_{D, \mathcal{M}}^{\ast}(\G)$. If $\mathcal{M}$ denotes the family of all quasi-invariant measures, we shall just say that the $D$-representation is \emph{faithful} when the above holds.

\begin{proposition} \label{prop: every D-C*-algebra has a faithful D-repr.}
	Let $D \trianglelefteq B(\G)$ be an algebraic ideal and $\mathcal{M}$ a family of quasi-invariant measures. Then $C_{D, \mathcal{M}}^{\ast}(\G)$ has a $D$-representation which is faithful with respect to $\mathcal{M}$.
	\begin{proof}
		Every ideal completion is a separable $\rm C^*$-algebra. Suppose, then, that $\{ a_m \}_m \subset C_{D, \mathcal{M}}^{\ast}(\G)$ is countable dense subset of elements. For every $m$ and $n$ we can find a $D$-representation $\pi_{m,n}$ such that $$ \sup_{\mu \in \mathcal{M}} \| (\pi_{m,n})_\mu (a_m) \| \geq \| a_m \|_{D, \mathcal{M}} - \frac{1}{n} .$$ By \cref{prop: direct sum of D-repr. is again a D-repr.}, the direct sum representation $\pi := \oplus_{m,n} \pi_{m,n}$ is a $D$-representation, and it follows from the above inequality that $ \sup_{\mu \in \mathcal{M}} \|\pi_\mu (a_m)\| = \|a_m\|_{D, \mathcal{M}} $, for every $m$. Let $a \in C_{D, \mathcal{M}}^{\ast}(\G)$ be given and find $a_{m_k} \to a$ in $C_{D, \mathcal{M}}^{\ast}(\G)$. Then $$ \sup_{\mu \in \mathcal{M}} \| \pi_\mu (a) - \pi_\mu (a_{m_k}) \| \leq \|a - a_{m_k}\|_{D, \mathcal{M}}\to 0 .$$ Since moreover $$ \|a_{m_k}\|_{D, \mathcal{M}} = \sup_{\mu \in \mathcal{M}} \|\pi_\mu (a_{m_k})\| \leq \sup_{\mu \in \mathcal{M}} \|\pi_{\mu} (a) - \pi_\mu (a_{m_k})\| + \sup_{\mu \in \mathcal{M}} \|\pi_\mu (a)\| ,$$ for all $k$, taking limits, we obtain that $$ \|a\|_{D, \mathcal{M}} \leq \sup_{\mu \in \mathcal{M}} \|\pi_\mu (a)\| \leq \|a\|_{D, \mathcal{M}} ,$$ proving the proposition.
	\end{proof}
\end{proposition}

One can adapt \cite[Theorem 1]{Haagerup1CowlingHowe:AlmostL^2MatrixCoefficients} to obtain the following highly useful result:
\begin{lemma} \label{lem: Bc-representations are weakly contained in regular reps.}
	Let $\sigma$ be a $B_c (\G)$-representation of $\G$. Then given any quasi-invariant measure $\mu$, the integrated form $\sigma_\mu$ is weakly contained in $\text{Ind}(\mu)$.
	\begin{proof}
		It suffices to prove that $$\|\sigma_\mu (f)\| \leq \|\text{Ind}(\mu)(f)\| ,$$ for all $f \in C_c (\G)$. Fix a self adjoint $f \in C_c (\G)$, and let $H_{\sigma , \mu} := \int_{\G^{(0)}}^{\oplus} H_\sigma (u) \, d \mu(u)$ denote the Hilbert space corresponding to $\sigma_\mu$. For any $\theta \in H_{\sigma , \mu}$, the continuous functional calculus applies to give a positive linear functional $$ g \mapsto \langle g(\sigma_\mu(f)) \theta , \theta \rangle ,$$ for $g \in C (\text{Sp}(\sigma_\mu (f)))$, and so there exists a positive measure $\mu_{\theta , \theta}$ such that $$ \int_{\text{Sp}(\sigma_\mu (f))} g(z) \, d \mu_{\theta , \theta} (z) = \langle g(\sigma_\mu(f)) \theta , \theta \rangle ,$$ for all $g \in C (\text{Sp}(\sigma(f)))$. By Hölder's inequality, $$ \int_{\text{Sp}(\sigma_\mu(f))} t^2 \, d \mu_{\theta , \theta} (t) \leq \left( \int_{\text{Sp}(\sigma_\mu(f))} t^{2n} \, d \mu_{\theta , \theta} \right)^{1/n} \left( \int_{\text{Sp}(\sigma_\mu(f))} 1 \, d \mu_{\theta , \theta} \right)^{1 - 1/n}  ,$$ and it is not hard to see that $$ \lim_{n \to \infty} \left( \int_{\text{Sp}(\sigma_\mu(f))} t^{2n} \, d \mu_{\theta , \theta} \right)^{1/n} = \sup \{ t^2 \colon t \in \supp (\mu_{\theta , \theta}) \} .$$ Hence $$ \| \sigma_\mu (f) \theta \| = \langle \sigma_\mu(f^{\ast 2}) \theta , \theta \rangle^{1/2} \leq \lim_{n \to \infty} \langle \sigma_\mu(f^{\ast 2n}) \theta , \theta \rangle^{1/2n} \| \theta \| ,$$ so that $$ \|\sigma_\mu(f)\| = \sup_{\theta \in K} \|\sigma_\mu(f) \theta\| \|\theta\|^{-1} \leq \sup_{\theta \in K} \lim_{n \to \infty} \langle \sigma_\mu(f^{\ast 2n}) \theta , \theta \rangle^{1/2n} ,$$ where $K$ is any dense subspace of $H_{\sigma , \mu}$. The converse inequality clearly holds. Thus, for any $f \in C_c (\G)$, we have 
		\begin{align*}
			\| \sigma_\mu (f) \| &=  \sup_{\theta \in K} \lim_{n \to \infty} \langle \sigma_\mu((f^{\ast} \ast f)^{\ast 2n}) \theta , \theta \rangle^{1/4n} \\ 
			&= \sup_{\theta \in K} \lim_{n \to \infty} \left( \int_{\G} (f^{\ast} \ast f)^{\ast 2n} (x) \langle \sigma(x) \theta(s(x)) , \theta(r(x)) \rangle_{H(r(x))} \, d \nu_0 (x) \right)^{1/4n} \\
			&\leq \sup_{\xi, \eta \in K} \liminf_{n \to \infty} \left| \int_{\G} (f^{\ast} \ast f)^{\ast 2n} (x) \langle \sigma(x) \xi(s(x)) , \eta(r(x)) \rangle_{H(r(x))} \, d \nu_0 (x) \right|^{1/4n}
		\end{align*}
		By \cref{lema: matrix coefficients in D for fundamental sequence imply the same for the dense subspace they span} there is a dense subspace $K$ of $H_{\sigma, \mu}$ such that for all $\xi , \eta \in K$, the map $$x \mapsto \langle \sigma (x) \xi (s(x)), \eta (r(x)) \rangle_{H(r(x))},$$ is in $ B_c (\G)$. Fix any such $\xi , \eta \in K$ and put $\psi (x) := \langle \sigma (x) \xi (s(x)), \eta (r(x)) \rangle_{H(r(x))}$. By Cauchy-Schwartz,
		\begin{align*}
			&\left| \int_{\G} (f^{\ast} \ast f)^{\ast 2n}(x) \langle \sigma (x) \xi (s(x)), \eta (r(x)) \rangle_{H(r(x))} \, d \nu_0 (x) \right|^2 \\ 
			&= \left| \int_{\G} (f^{\ast} \ast f)^{\ast 2n}(x) \psi(x) \, d \nu_0 (x) \right|^2 \\
			&\leq \left( \int_{\G} |( f^{\ast} \ast f )^{\ast 2n} (x) | |\psi(x)| \Delta^{1/2}(x) \, d \nu^{-1} (x) \right)^2 \\ 
			&\leq \| (f^{\ast} \ast f)^{\ast 2n} \|_{L^2 (\G, \nu^{-1})}^{2} \, \int_{\G} |\psi|^2(x) \Delta(x) d \nu^{-1} 
			= \| (f^{\ast} \ast f)^{\ast 2n} \|_{L^2 (\G, \nu^{-1})}^{2} \, \| \psi \|_{L^2 (\G, v)}^{2}.
		\end{align*}
		Thus,
		\begin{gather*}
			\liminf_{n \to \infty} \left| \int_{\G} (f^{\ast} \ast f)^{\ast 2n}(x) \langle \sigma (x) \xi (s(x)), \eta (r(x)) \rangle_{H(r(x))} \, d \nu_0 (x) \right|^{1/4n} \\
			\leq \liminf_{n \to \infty} \| (f^{\ast} \ast f)^{\ast 2n} \|_{L^2 (\G, \nu^{-1})}^{1/4n},
		\end{gather*}
		so that $$ \| \sigma_\mu (f) \| \leq \liminf_{n \to \infty} \| (f^{\ast} \ast f)^{\ast 2n} \|_{L^2 (\G, \nu^{-1})}^{1/4n} .$$ A similar estimate holds for $\text{Ind}(\mu)$. Indeed, let $\lambda$ denote the left regular representation as usual. We may take $C_c (\G)$ as the dense subspace of $\int_{\G^{(0)}}^{\oplus} \ell^2 (\G^u) \, d \mu(u)$. For $g,h \in C_c (\G)$, the map $x \mapsto \langle \lambda(x) g,h \rangle_{\ell^{2}(\G^{r(x)})}$ is in $C_c (\G)$, so that by Cauchy-Schwartz again,
		\begin{gather*}
			\left| \int_{\G} (f^{\ast} \ast f)^{\ast 2n}(x) \langle \lambda(x) g,h \rangle_{\ell^{2}(\G^{r(x)})} \, d \nu_0 (x) \right|^{2} \\
			\leq \| (f^{\ast} \ast f)^{\ast 2n} \|_{L^2 (\G, \nu^{-1})}^{2} \, \| \langle \lambda (\cdot) g, h \rangle_{\ell^2 (\G^{r(\cdot)})} \|_{L^2 (\G, \nu)}^{2}
		\end{gather*} 
		so we have also that $$ \| \text{Ind}(\mu) (f) \| = \| \lambda_\mu (f) \| \leq \liminf_{n \to \infty} \|(f^{\ast} \ast f)^{\ast 2n} \|_{L^2 (\G, \nu^{-1})}^{1/4n} .$$ Since $$(f^{\ast} \ast f)^{\ast 2n} = (f^{\ast} \ast f)^{\ast (2n - 2)} \ast (f^{\ast} \ast f)^{\ast 2} ,$$ we have that 
		\begin{align*}
			\| (f^{\ast} \ast f)^{\ast 2n} \|_{L^2 (\G, \nu^{-1})} &= \|(f^{\ast} \ast f)^{\ast (2n - 2)} \ast (f^{\ast} \ast f)^{\ast 2}\|_{L^2 (\G, \nu^{-1})} \\
			&\leq \| \text{Ind}(\mu) ((f^{\ast} \ast f)^{\ast (2n - 2)}) \| \|(f^{\ast} \ast f)^{\ast 2}\|_{L^2 (\G, \nu^{-1})} 
		\end{align*}
		and hence $$ \limsup_{n \to \infty} \| (f^{\ast} \ast f)^{\ast 2n} \|_{L^2 (\G, \nu^{-1})}^{1/4n} \leq \| \text{Ind}(\mu)(f) \| .$$ In summary, $$ \limsup_{n \to \infty} \|(f^{\ast} \ast f)^{\ast 2n}\|_{L^2 (\G, \nu^{-1})}^{1/4n} \leq \|\text{Ind}(\mu)(f)\| \leq \liminf_{n \to \infty} \|(f^{\ast} \ast f)^{\ast 2n}\|_{L^2 (\G, \nu^{-1})}^{1/4n} ,$$ so that the limit of
		$ \|(f^{\ast} \ast f)^{\ast 2n}\|_{L^2 (\G, \nu^{-1})}^{1/4n}$ exists and equals $\| \text{Ind}(\mu)(f) \|$. Therefore, $$ \| \sigma_\mu(f) \| \leq \lim_{n \to \infty} \| (f^{\ast} \ast f)^{\ast 2n} \|_{L^2 (\G, \nu^{-1})}^{1/4n} = \| \text{Ind}(\mu)(f) \| .$$ This proves the lemma.
	\end{proof}
\end{lemma}

\begin{remark} \label{rmk: in Haagerup estimate, it suffices to consider L2-representations}
	Let $\sigma$ be a unitary representation and $\mu$ a quasi-invariant measure. Notice that by the proof of \cref{lem: Bc-representations are weakly contained in regular reps.}, for the integrated form $\sigma_\mu$ to be weakly contained in $\text{Ind}(\mu)$, it actually suffices to assume that $\sigma$ is an $L^2 (\G, \nu_\mu)$-representation. 
\end{remark}

\begin{proposition} \label{prop: reduced C*-algebra is the one associated with the ideal B_c (G)}
	$\Cred(\G) = C_{B_c (\G), \mathcal{M}}^{\ast}(\G)$, where $\mathcal{M}$ is any family of quasi-invariant measures with full support. In particular, $\Cred(\G) = C_{B_c (\G)}^{\ast}(\G)$.
	\begin{proof}
		Let $\rho$ be a $B_c (\G)$-representation of $\G$. By \cref{lem: Bc-representations are weakly contained in regular reps.}, for any $\mu \in \mathcal{M}$, the representation $\rho_\mu$ is weakly contained in $\text{Ind}(\mu)$. Thus, for any $f \in C_c (\G)$, we have $$\|f\|_{r} \geq \|\text{Ind}(\mu)(f)\| \geq \|\rho_\mu (f)\|,$$ and so $$ \|f\|_{r} \geq \|f\|_{B_c (\G) , \mathcal{M}} ,$$ for all $f \in C_c (\G)$. Since $\mathcal{M}$ has full support and the left regular representation is a $B_c (\G)$ representation, we also have the reverse inequality.
	\end{proof}
\end{proposition}

\section{Haagerup property and positive definite functions} \label{sec: Haagerup property and positive definite functions}
A definition of Haagerup property for second-countable étale groupoids has been given in \cite[Proposition 5.4 (2)]{KwasniewskiLiSkalski:TheHaagerupPropertyForTwistedGrouopidDynamicalSystems} (see also \cite[Definition 6.2]{Delaroche:HaagerupPropertyForMeasuredGroupoids} for a definition for countable measured groupoids). This definition generalizes the corresponding one for countable discrete groups. In this section, we investigate the connection between this groupoid Haagerup property and the theory developed in \cref{sec: constructing grouopid c*-algebras}. We shall begin by recalling the definition of a positive definite function on the groupoid, and then move on to define the Haagerup property and provide some examples of groupoids having this property. Subsequently, we shall introduce the necessary tools regarding positive definite functions on groupoids needed for proving the main result of this section, namely \cref{thm: Haagerup property implies the B_0 - C*-algebra coincides with the maximal one}. Throughout this section, all groupoids are assumed to be Hausdorff second-countable and étale. 

\begin{definition} \label{def: positive definite function on groupoid}
	Let $F \colon \G \to \C$ be a Borel function. We say that $F$ is \emph{positive definite} if for any $u \in \G^{(0)}$, $x_1 , \dots, x_n \in \G^u$ and complex numbers $\alpha_1 , \dots, \alpha_n$, we have that $$ \sum_{i,j = 1}^{n} \overline{\alpha_i} \alpha_j F(x_{i}^{-1} x_j) \geq 0 .$$ Equivalently, if for any $u \in \G^{(0)}$ and $x_1 , \dots, x_n \in \G^u$, the matrix $[F(x_{i}^{-1} x_j)]_{i,j}$ is positive semi-definite.
\end{definition}
It is worthwhile to note that a positive definite function $F$ satisfies $\sup_{x \in \G} | F(x) | \leq \sup_{u \in \G^{(0)}} |F (u)|$ and $F(x^{-1}) = \overline{F(x)}$, for all $x \in \G$.

Recall that for a subset $A \subset \G^{(0)}$, $\G(A)$ denotes the restriction subgroupoid given by all elements of $\G$ whose range and source lies in $A$. A continuous function $f$ is said to be \emph{locally $C_0$} provided that for every compact $K \subset \G^{(0)}$, the appropriate restriction of $f$ lies in $C_0 (\G(K))$.
\begin{definition} \label{def: Haagerup property}
	An étale groupoid $\G$ is said to have the \emph{Haagerup property} if there exists a sequence of continuous positive definite functions $\{ F_n \}_{n \in \N}$ such that the following holds:
	\begin{itemize}
		\item[(1)] $F_n (u) = 1$, for all $u \in \G^{(0)}$ and $n \in \N$;
		\item[(2)] $\lim_{n \to \infty} F_n = 1$ uniformly on compact sets;
		\item[(3)] Each $F_n$ is locally $C_0$.
	\end{itemize}  
\end{definition}

Notice that the Haagerup property passes to closed subgroupoids, and that when $\G^{(0)}$ is compact, condition (3) of \cref{def: Haagerup property} reduces to each $F_n$ being in $C_0 (\G)$.
As is the case for groups, we expect amenability (see \cite{DelarocheRenault:AmenableGroupoids}) to be stronger than the Haagerup property, at least for certain classes of groupoids. Indeed, \cite[Proposition 5.4 (3)]{KwasniewskiLiSkalski:TheHaagerupPropertyForTwistedGrouopidDynamicalSystems} shows that $\G$ has the Haagerup property whenever $\G$ is amenable such that the unit space can be written as a union of compact open sets (for example the unit space is totally disconnected or compact).

\begin{example} \label{ex: examples of groupoids with Haagerup property}
	\ 
	\begin{itemize}
		\item[(1)] Any countable discrete group with the Haagerup property has the groupoid Haagerup property.
		\item[(2)] Suppose $\Gamma$ is a countable discrete group acting on a locally compact Hausdorff space $X$. If $\Gamma$ has the Haagerup property, then the transformation groupoid $\Gamma \rtimes X$ has the Haagerup property. Indeed, if $\{f_n \}_n$ is a sequence of positive definite $c_0 (\Gamma)$ functions satisfying the conditions in \cref{def: Haagerup property}, then it is easily checked that the functions $F_n (\gamma,x) = f_n (\gamma)$, for any $(\gamma,x) \in \Gamma \rtimes X$, are locally $C_0$ functions satisfying the conditions in \cref{def: Haagerup property}.
		\item[(3)] Every amenable second-countable étale groupoid with compact unit space has the Haagerup property.
	\end{itemize}
\end{example}

We let $B_{0,l}(\G)$ denote the subspace of all bounded Borel functions which are locally $B_0$; that is, $B_{0,l}(\G)$ consists of functions $f \in B(\G)$ such that for all compact $K \subset \G^{(0)}$, we have $f|_{\G(K)} \in B_0 (\G(K))$. Clearly the set $B_{0,l}(\G) \subset B(\G)$ is an ideal containing $B_0 (\G)$. Similarly, we let $C_{0,l}(\G)$ denote the continuous functions on $\G$ which are locally $C_0$.

Like for groups, positive definite functions on groupoids are intimately related to unitary representations of groupoids. A result covering a more general case than that of the next proposition, can be found in \cite[Lemma 3.2]{RamsayAndWalter:FourierStieltjesAlgebrasOfLocallyCompactGroupoids}. 

\begin{proposition} \label{prop: bounded Borel sections and unitaries gives positive definite functions}
	Let $\pi$ be a unitary representation of $\G$. Fix any Borel section $\xi \in S(\G^{(0)} \ast H_\pi)$, and let $F(x) = \langle \pi(x) \xi(s(x)) , \xi(r(x)) \rangle_{H(r(x))}$. Then $F$ is a positive definite function on $\G$. If $\xi$ is bounded, then so is $F$.
	\begin{proof}
		First we show that $F$ is Borel. This is already very well known, but we include an argument here for convenience. Let $\{f_n\}_n$ be an orthonormal fundamental sequence for the Borel Hilbert bundle $\G^{(0)} \ast H_\pi$. Thus, for each $x \in \G$, we may write $$ \xi(s(x)) = \sum_{m = 1}^{\infty} \langle \xi(s(x)) , f_m (s(x)) \rangle_{H_\pi (s(x))} f_m (s(x)) ,$$ and so $$ \pi(x) \xi(s(x)) = \sum_{m = 1}^{\infty} \langle \xi(s(x)) , f_m (s(x)) \rangle_{H_\pi (s(x))} \pi(x) f_m (s(x)).$$ If $n \in \N$, then 
		\begin{align*}
			x \mapsto &\langle \pi(x) \xi(s(x)) , f_n (r(x)) \rangle_{H_\pi (r(x))} \\ 
			&= \sum_{m = 1}^{\infty} \langle \xi(s(x)) , f_m (s(x)) \rangle_{H_\pi (s(x))} \langle \pi(x) f_m (s(x)) , f_n (r(x)) \rangle_{H_\pi (r(x))} ,	
		\end{align*}
		is a Borel function since for every $m \in \N$, both $$ x \mapsto \langle \xi(s(x)) , f_m (s(x)) \rangle_{H_\pi (s(x))} ,$$ and $$x \mapsto \langle \pi(x) f_m (s(x)) , f_n (r(x)) \rangle_{H_\pi (r(x))} ,$$ are Borel by \cref{lem: measurability of maps into Borel hilbert bundle}, \cref{def: isomorphism groupoid PaperA} and \cref{def: unitary representation of groupoids PaperA}. Using this together with \cref{lem: measurability of maps into Borel hilbert bundle} again, we obtain that 
		\begin{align*}
			x \mapsto F(x) &= \langle \pi(x) \xi(s(x)) , \xi(r(x)) \rangle_{H_\pi (r(x))} \\
			&= \sum_{m = 1}^{\infty} \langle \pi(x) \xi(s(x)) , f_m(r(x)) \rangle_{H_\pi (r(x))} \overline{\langle \xi(r(x)), f_m (r(x)) \rangle}_{H_\pi (r(x))} ,
		\end{align*}
		is Borel.
		
		Let $u \in \G^{(0)}$, $x_1 , \ldots , x_n \in \G^u$, and $\alpha_1 , \ldots , \alpha_n \in \C$. Put $\eta := \sum_{i = 1}^{n} \alpha_i \pi(x_i) \xi(s(x_i)) \in H_\pi (u)$. We have that
		\begin{align*}
			\sum_{i,j = 1}^{n} \overline{\alpha_i} \alpha_j F(x_{i}^{-1} x_j) &= \sum_{i,j = 1}^{n} \overline{\alpha_i} \alpha_j \langle \pi(x_{i}^{-1} x_j) \xi(s(x_j)) , \xi(s(x_i)) \rangle_{H_\pi (s(x_i))} \\
			&= \sum_{i,j = 1}^{n} \overline{\alpha_i} \alpha_j \langle \pi(x_j) \xi(s(x_j)) , \pi(x_i) \xi(s(x_i)) \rangle_{H_\pi (u)} \\
			&= \left\langle \sum_{i = 1}^{n} \alpha_i \pi(x_i) \xi(s(x_i)) , \sum_{i = 1}^{n} \alpha_i \pi(x_i) \xi(s(x_i)) \right\rangle_{H_\pi (u)} = \left\| \eta \right\|_{H_\pi (u)}^{2} \geq 0.
		\end{align*}
		The final statement follows by the inequality 
		\[ \sup_{x \in \G} | F(x) | \leq \sup_{u \in \G^{(0)}} |F (u)| = \sup_{u \in \G^{(0)}} \| \xi (u) \|_{H_\pi (u)}^{2} < \infty \qedhere\]
	\end{proof}
\end{proposition}

Conversely, \cite[Lemma 3.3 and Theorem 3.5]{RamsayAndWalter:FourierStieltjesAlgebrasOfLocallyCompactGroupoids} gives us the following.
\begin{proposition} \cite[Lemma 3.3 and Theorem 3.5]{RamsayAndWalter:FourierStieltjesAlgebrasOfLocallyCompactGroupoids} \label{prop: positive definite function gives rise to a unitary representation}
	Let $F$ be a bounded positive definite function on a groupoid $\G$. Fix any unit $u \in \G^{(0)}$ and let $\mathcal{H}(u)$ denote the pre-inner product space consisting of all finitely supported functions on $\G^u$, where the pre-inner product is given by $$ \left\langle f, g \right\rangle_u = \sum_{x,y \in \G^u} f(x) \overline{g(y)} F(y^{-1}x) .$$ Let $H(u)$ be the Hilbert space obtained by separation and completion. Define a representation by $$\pi_F (x) \colon H(s(x)) \to H(r(x)) \, , \, \pi_F (x)f (y) = f(x^{-1} y),$$ for $ x \in \G$. One can endow the bundle $\G^{(0)} \ast H$ with the structure of a Borel Hilbert bundle with fundamental sequence given by a countable subset of $C_c (\G)$ that is dense for the inductive limit topology. Then $\pi_F \colon \G \to \text{Iso}(\G^{(0)} \ast H)$ becomes a unitary representation. 
	There exists a bounded Borel section $\xi_F \colon u \mapsto \xi_F(u) \in H(u)$, such that $ F(x) = \langle \pi_F (x) \xi_F (s(x)), \xi_F (r(x)) \rangle_{H(r(x))} ,$ $\nu_\mu$-almost everywhere for any quasi-invariant measure $\mu$. If for any $f \in C_c (\G)$ we denote by $\sigma(f)$ the section given by restricting $f$ to the fibers, then $\pi{_{F}}_{\mu} (f) \xi_F = \sigma(f)$ in $H_{\pi_F , \mu}$. 
	Moreover, if $F$ is continuous, then $\xi_F$ can be chosen so that $F(x) = \left\langle \pi_F (x) \xi_F (s(x)) , \xi_F (r(x)) \right\rangle_{H(r(x))}$ for all $x \in \G$. 
\end{proposition}

We shall call the representation induced by $F$ as above the \emph{GNS representation} associated to $F$. In order to prove the main result of this section, we shall need the following technical lemma.
\begin{lemma} \label{lem: GNS representation associated with positive definite B_0 function is B_0 representation}
	Let $F \colon \G \to \C$ be a continuous positive definite function. If $F \in B_{0,l} (\G)$, then the associated GNS representation is a $B_{0,l} (\G)$-representation.
	\begin{proof}
		Let $\pi_F$ be the GNS representation induced from $F$, and let $\xi_F$ be a bounded Borel section such that $F(x) = \langle \pi_F (x) \xi_F (s(x)) , \xi_F (r(x)) \rangle_{H(r(x))}$, for every $x \in \G$. Given two fundamental sections $f,g \in C_c (\G)$, put $\psi (x) := \langle \pi_F (x) \sigma(f)(s(x)) , \sigma(g)(r(x)) \rangle_{H(r(x))}$, where we recall that $\sigma(f)$ is the section associated to $f$ given by restriction. We shall show that $\psi \in B_{0,l} (\G)$. If this is true, then $\pi_F$ is indeed a $B_{0,l} (\G)$-representation.
		We compute 
		\begin{align*}
			\psi(x) &= \langle \pi_F (x) \sigma(f)(s(x)) , \sigma(g)(r(x)) \rangle_{H(r(x))} \\
			&= \sum_{y,z \in \G^{r(x)}} f(x^{-1}z) \overline{g(y)} F(y^{-1}z) \\
			&= \sum_{z \in \G^{r(x)}} f(x^{-1}z) (\overline{g} \ast F)(z) \\ 
			&= \sum_{z \in \G^{r(x)}} \overline{f}^{\ast} (z^{-1}x) (\overline{g} \ast F)(z) \\
			&= (\overline{g} \ast F) \ast \overline{f}^{\ast} (x).
		\end{align*} 
		Since any continuous compactly supported function on $\G$ may be written as a sum of continuous functions supported on bisections, it suffices to show that for such $f,g$, the function $\bar{g} \ast F \ast \bar{f}^{\ast} $ is in $B_{0,l} (\G)$. 
		
		We shall show next that for $g \in C_c (\G)$ supported on a single bisection and $f \in B_{0,l}(\G)$, we have that $g \ast f \in B_{0,l}(\G)$. To this end, let $K \subset \G^{(0)}$ be compact. We wish to see that $g \ast f \in B_0 (\G(K))$. In any case, given $x \in \G(K)$, we have $g \ast f (x) = g(y) f(y^{-1}x)$, when there exists $y \in \supp (g)$ with $r(y) = r(x)$, and is zero otherwise. Since $f \in B_0 (\G(K))$, given $\epsilon > 0$, there exists a compact set $C \subset \G(K)$ such that if $z \notin C$, then $|f(z)| \leq \epsilon/ \| g \|_{\infty}$. Since $\supp (g)$ is compact and multiplication is continuous, $K \supp (g) C$ is compact. Assume that $x \in \G(K) \setminus K \supp (g) C $. If there is no $y \in \supp(g)$ such that $r(y) = r(x)$, then $g \ast f (x) = 0$. So suppose there exist $y \in \supp(g)$ with $r(y) = r(x)$. Then we must have $y^{-1} x \notin C$; for otherwise $x \in K yC \subset K \supp(g) C$ which is a contradiction; whence $| f(y^{-1}x) | \leq \epsilon / \| g \|_{\infty}$, and so $|g \ast f (x) | = | g(y) f(y^{-1}x) | \leq \| g \|_{\infty} \epsilon / \| g \|_{\infty} = \epsilon $, whenever $x \notin K \supp(g)C$.
		
		Clearly, $f \in B_{0,l}(\G)$ if and only if $f^{\ast} \in B_{0,l}(\G)$. It follows from this together with the above argument that if instead $g \in B_{0,l}(\G)$ and $f \in C_c (\G)$ supported on a single bisection, then $g \ast f \in B_{0,l}(\G)$. Indeed, in this case $(g \ast f)^{\ast} = f^\ast \ast g^\ast \in B_{0,l}(\G)$ by the above, and by what has just been remarked, we need $g \ast f \in B_0 (\G)$. 
		
		Now, applying the above twice to the function $\psi (x) = (\bar{g} \ast F) \ast \overline{f}^{\ast} (x)$, we see that it is indeed in $B_{0,l} (\G)$.
	\end{proof}
\end{lemma}

\begin{theorem} \label{thm: Haagerup property implies the B_0 - C*-algebra coincides with the maximal one}
	If $\G$ has the Haagerup property, then $C_{B_{0,l} (\G)}^{\ast} (\G) = C^{\ast}(\G)$.
	\begin{proof}
		Suppose $\G$ has the Haagerup property. Then there exists a sequence of continuous positive definite functions $\{ F_n \}_n \subset C_{0,l} (\G)$ such that $F_n \to 1$ uniformly on compact sets and $F_n |_{\G^{(0)}} = 1$ for all $n \in \N$. Fix any unitary representation $\pi$ and quasi-invariant measure $\mu$. Let $\pi_n$ denote the GNS representations corresponding to the $F_n$, with cyclic vectors $\xi_n$, as in \cref{prop: positive definite function gives rise to a unitary representation}. By \cref{lem: GNS representation associated with positive definite B_0 function is B_0 representation}, the unitary representation $\pi_n$ are $B_{0,l} (\G)$-representations. We shall show that $\|\pi_\mu (a)\| \leq \|a\|_{B_{0,l} (\G)}$, for all $a \in C^* (\G)$, of course, using the representations $\pi_n$. The unitary representations given by the tensor product $\pi_n \otimes \pi$ are $B_{0,l} (\G)$-representations by \cref{lem: tensor product of a D-representation with another is a D-representation}, and for any $\eta \in H_{\pi , \mu}$, we have 
		\begin{align*}
			&\langle \pi_n \otimes \pi (x) (\xi_n \otimes \eta)(s(x)) , \xi_n \otimes \eta (r(x)) \rangle_{H_{\pi_n} (r(x)) \otimes H_\pi (r(x))} \\ 
			&= \langle \pi_n (x) (\xi_n (s(x))) , \xi_n (r(x)) \rangle_{H_{\pi_n}(r(x))} \langle \pi(x) \eta(s(x)), \eta(r(x)) \rangle_{H_\pi (r(x))} \\ 
			&= F_n (x) \langle \pi(x) \eta(s(x)), \eta(r(x)) \rangle_{H_\pi (r(x))} ,
		\end{align*} 
		for $x \in \G$. Therefore, given any $f \in C_c (\G)$ and $\eta \in H_{\pi , \mu}$, we have that 
		\begin{align*}
			&\left\langle (\pi_n \otimes \pi)_\mu (f) (\xi_n \otimes \eta) , (\xi_n \otimes \eta) \right\rangle \\
			&= \int_{\G} f(x) \langle \pi_n (x) (\xi_n (s(x))) , \xi_n (r(x)) \rangle_{H_{\pi_n}(r(x))} \langle \pi(x) \eta(s(x)), \eta(r(x)) \rangle_{H_\pi (r(x))} \, d \nu_0 (x) \\
			&= \int_{\G} f(x) F_n (x) \langle \pi(x) \eta (s(x)) , \eta (r(x)) \rangle_{H_\pi (r(x))} \, d \nu_0 (x).
		\end{align*}
		The Lebesgue dominated convergence theorem gives us that $$ \lim_{n \to \infty} \left\langle (\pi_n \otimes \pi)_\mu (f) (\xi_n \otimes \eta) , \xi_n \otimes \eta \right\rangle = \int_{\G} f(x) \langle \pi(x) \eta (s(x)) , \eta (r(x)) \rangle_{H_\pi (r(x))} \, d \nu_0 (x) .$$ That is, given any $f \in C_c (\G)$ and $\eta \in H_{\pi,\mu}$, we have $$ \lim_{n \to \infty} \left\langle (\pi_n \otimes \pi)_\mu(f) (\xi_n \otimes \eta) , \xi_n \otimes \eta \right\rangle = \left\langle \pi_\mu (f) (\eta) , \eta \right\rangle .$$ 
		Since $1 = F_n (u) = \|\xi_n (u)\|_{H(u)}^{2}$, for all $u \in \G^{(0)}$, given any $\eta \in H_{\pi , \mu}$, we have that
		\begin{align*}
			\|\xi_n \otimes \eta\|^2 &= \int_{\G^{(0)}} \|\xi_n \otimes \eta (u)\|_{H_{\pi_n}(u) \otimes H_\pi (u)}^{2} \, d \mu(u) \\
			&= \int_{\G^{(0)}} \|\xi_n (u)\|_{H_{\pi_n} (u)}^{2} \|\eta (u)\|_{H_\pi (u)}^{2} \, d \mu(u) \\
			&= \int_{\G^{(0)}} \|\eta (u)\|_{H_\pi (u)}^{2} \, d \mu(u) = \|\eta\|^2 .
		\end{align*}
		Using this together with the fact that $C_c (\G)$-functions are dense in $C^{\ast}(\G)$, one can show by a standard $\epsilon/3$-argument that given any $a \in C^{\ast}(\G)$ and $\eta \in H_{\pi , \mu}$, we have $$ \lim_{n \to \infty} \left\langle (\pi_n \otimes \pi)_\mu(a) (\xi_n \otimes \eta) , \xi_n \otimes \eta \right\rangle = \langle \pi_\mu (a) (\eta) , \eta \rangle . $$ 
		Now, it follows by \cite[Theorem 3.4.4]{Dixmier:C*-algebras} that $$ \bigcap_{n \in \N} \ker ((\pi_n \otimes \pi)_\mu) \subset \ker(\pi_\mu) ,$$ so that $$ \left\|\pi_\mu (a)\right\| \leq \left\|\oplus_{n \in \N} (\pi_n \otimes \pi)_\mu (a)\right\| = \sup_{n \in \N} \left\|(\pi_n \otimes \pi)_\mu (a)\right\| ,$$ for all $a \in C^{\ast} (\G)$. Hence $\left\|\pi_\mu (a)\right\| \leq \left\|a\right\|_{B_{0,l} (\G)}$, for all $a \in C^{\ast} (\G)$. Since the unitary representation $\pi$ and the quasi-invariant measure $\mu$ was arbitrary, we can conclude that $ C_{B_{0,l} (\G)}^{\ast}(\G) = C^{\ast}(\G) $.
	\end{proof}
\end{theorem}

\begin{remark} \label{rmk: amenability implies weak containment}
	Suppose $\G$ is amenable. Then with the obvious modifications to the argument in \cref{thm: Haagerup property implies the B_0 - C*-algebra coincides with the maximal one}, we get that $C_{B_c(\G)}^{\ast} (\G) = C^{\ast}(\G)$. Using this together with \cref{prop: reduced C*-algebra is the one associated with the ideal B_c (G)}, then gives the classical statement that $C_{r}^{\ast}(\G) = C^{\ast}(\G)$, when $\G$ is amenable.
\end{remark}

\begin{remark} \label{rmk: unable to prove converse statement to Haagerup main theorem}
	For countable groups, the converse statement to \cref{thm: Haagerup property implies the B_0 - C*-algebra coincides with the maximal one} is true (see \cite[Corollary 3.4]{BrownandGuentner:NewC*-completions}), but we have not been able to prove this for Hausdorff second-countable étale groupoids. It may very well be that the converse is false, just as the analogous statement for amenability and these completions is false (see \cite{AlekseevAndFinnSell:NonAmenablePrincipalGroupoidsWithWeakContainment, Willett:ANonAmenableGroupoidWhoseMaximalAndReducedC*AlgebrasAreTheSame}). 
\end{remark}

\section{Exotic groupoid $\rm C^*$-algebras associated to hyperbolic groupoids} \label{sec: exotic groupoid C*-algebras associated to hyperbolic groupoids}

In this section, we use the theory developed in \cref{sec: constructing grouopid c*-algebras} to construct exotic groupoid $\rm C^*$-algebras associated to certain metrically hyperbolic groupoids (see \cref{def: hyperbolic groupoids}). Following \cite[Definition 1.1]{ChristensenAndNeshveyev:IsotropyFibersOfIdealsInGroupoidCstarAlgebras}, let us define precisely what we mean by a (exotic) groupoid $\rm C^*$-algebra.
\begin{definition} \label{def: exotic groupoid C*-algebra}
	Let $\G$ be a second-countable Hausdorff étale groupoid. By a \emph{groupoid $\rm C^*$-algebra} for $\G$, we mean a $\rm C^*$-algebra $C_{\mathpzc{e}}^{\ast} (\G)$ which is the completion of $C_c (\G)$ under some $\rm C^*$-norm $\|\cdot\|_{\mathpzc{e}}$ which dominates the reduced norm. If the norm $\|\cdot\|_{\mathpzc{e}}$ is neither equal to the full nor the reduced norm, then we call it \emph{exotic} and refer to the associated completion as an \emph{exotic groupoid $\rm C^*$-algebra}.
\end{definition}
Thus, for an exotic groupoid $\rm C^*$-algebra the identity map on $C_c (\G)$ extend to surjections which are non-injective $$ C^* (\G) \twoheadrightarrow C_{\mathpzc{e}}^{\ast} (\G) \twoheadrightarrow C_{r}^{*} (\G) .$$ 
We show in this section that certain metrically hyperbolic groupoids do admit (many) exotic groupoid $\rm C^*$-algebras. The proof strategy for this is similar to that of Okayasu in \cite{Okayasu:FreeGroupC*-algebrasAssociatedWithLP}, and consists of analyzing for a certain class of positive definite functions on the groupoid, which naturally associated linear functionals on $C_c (\G) $ can and cannot be extended to states on certain groupoid $\rm C^*$-algebras. 

A definition of hyperbolic groupoids can be found in the book \cite{Nekrashevych:HyperbolicGroupoidsAndDuality} by Nekrashevych. Interestingly, to each such hyperbolic groupoid $\G$ there is a naturally defined dual groupoid $\G^{\top}$, which is itself hyperbolic and such that $(\G^{\top})^{\top}$ is equivalent to $\G$. These hyperbolic groupoids are Hausdorff groupoids of germs satisfying several conditions like being compactly generated, the Cayley graphs associated with each unit are Gromov-hyperbolic, and the germs satisfying a local Lipschitz condition. The definition in \cite{Nekrashevych:HyperbolicGroupoidsAndDuality}, however, does not incorporate hyperbolic groups seen as groupoids, because groupoids of germs are effective whilst groups seen as groupoids are not. In this subsection, we are interested in proving the existence of exotic groupoid $\rm C^*$-algebras for groupoids which are compactly generated and whose Cayley graphs are Gromov-hyperbolic. We shall refer to such groupoids as metrically hyperbolic. In order to state their definition, we need some preliminaries. 

A \emph{generating set} for an étale groupoid $\G$ is a compact open subset $S \subset \G$ such that $S^{-1} = S$ and $\G = \bigcup_{n \geq 1} S^n$. If a groupoid admits such a generating set, then it follows that $s (S) = r(S) = \G^{(0)}$ and hence the unit space is necessarily compact. Moreover, the groupoid then becomes equipped with a continuous length function $l_S \colon \G \to \R_+$, defined by $l_S (\gamma) = \min \{n \in \N_0 \colon \gamma \in \bigcup_{k = 0}^{n} S^k \}$, where $S^0 := \G^{(0)}$. The generating set defines the structure of a graph on each range fiber, called the \emph{Cayley graph} defined as follows: Let $u \in \G^{(0)}$ be a unit. The Cayley graph $\G^{u}$ is the graph whose vertices are elements $\gamma \in \G^{u}$ and there is an edge from $\gamma \in \G^{u}$ to $\eta \in \G^{u}$ if and only if there is $s \in S$ such that $\gamma = \eta s$. The Cayley graphs $\{\G^u \}_{u \in \G^{(0)}}$ come equipped with a field of metrics $\{d_u \}_{u \in \G^{(0)}}$, where $d_{r(\gamma)} (\gamma, \eta )$ is the minimum number of edges in a path connecting the vertex $\gamma$ to the vertex $\eta$. Thus, $d_{r(\gamma)}(\gamma, \eta) = l_S (\gamma^{-1} \eta)$, and from this it follows easily that the field of metrics are left invariant in the sense that whenever $(\xi , \gamma) \in \G^{(2)}$, then $d_{r(\xi)} (\xi \gamma, \xi \eta) = d_{r(\gamma)} (\gamma, \eta)$.

\begin{definition} \label{def: hyperbolic groupoids}
	A second-countable Hausdorff étale groupoid $\G$ is \emph{metrically hyperbolic} if there is a compact open generating set $S$ such that with the canonically induced metrics, the metric spaces $(\G^u , d_u)$ are all hyperbolic with same hyperbolicity constant; that is, there exists a $\delta > 0$ such that for any $u \in \G^{(0)}$ and $\gamma, \eta, \xi, \omega \in \G^{u}$, we have $$ d_u (\gamma, \eta) + d_u (\xi, \omega) \leq \max\{ d_u (\gamma, \xi) + d_u (\eta, \omega) , d_u (\gamma, \omega) + d_u (\eta, \xi) \} + \delta .$$
\end{definition}

Going forward, we shall often omit the unit-subscript on the metrics $\{d_u\}_u$, so that when we write $d(\gamma, \xi)$ it is clear that $r(\gamma) = r(\xi)$ and $d(\gamma, \xi) = d_{r(\gamma)} (\gamma, \xi) = l_S (\gamma^{-1} \xi)$. Notice that $d(r(\gamma), \gamma) = l_S (\gamma)$. We shall sometimes refer to $l_S$ as the canonically associated length function.

\begin{example} \label{ex: examples of hyperbolic groupoids}
	It is not hard to see that the following groupoids are metrically hyperbolic:
	\begin{itemize}
		\item Finitely generated hyperbolic groups;
		\item Graph groupoids associated with finite directed graphs;
		\item Transformation groupoids $\Gamma \rtimes X$, where $X$ is compact and $\Gamma$ is a finitely generated hyperbolic group;
		\item Treeable groupoids as in \cite[Section 7]{Delaroche:HaagerupPropertyForMeasuredGroupoids} where the groupoid therein is taken to be étale and the generating set to be compact open.
	\end{itemize}
\end{example}

Recall that a (real) \emph{conditionally negative definite} function on a groupoid $\G$ is a map $\psi \colon \G \to \R$ such that $\psi (x^{-1}) = \psi (x)$ for every $x \in \G$, $\psi(u) = 0$ for all $u \in \G^{(0)}$ and whenever $\{x_i \}_{i = 1}^{N}$ is a collection of elements in $\G$ with same range and $\{ \alpha_i \}_{i = 1}^{N}$ is collection of real numbers with $\sum_{i = 1}^{N} \alpha_i = 0$, then $$ \sum_{i,j = 1}^{N} \alpha_i \alpha_j \psi(x_{i}^{-1} x_j)  \leq 0.$$

If $\mu$ is an invariant measure on the unit space of an étale groupoid $\G$, and $p \in [1, \infty)$, then we let $\mathcal{L}_{\mu}^{p} (\G)$ denote the algebraic ideal in $B(\G)$ consisting of all bounded Borel functions $f$ such that $$ \int_\G |f(x)|^p \, d \nu(x) = \int_{\G^{(0)}} \sum_{x \in \G^u} |f(x)|^p \, d \mu(u) = \int_{\G^{(0)}} \sum_{x \in \G_u} |f(x)|^p \, d \mu(u) < \infty .$$ Clearly $\mathcal{L}_{\mu}^{q}(\G) \subset \mathcal{L}_{\mu}^{p}(\G)$ whenever $1 \leq q \leq p < \infty$.
The rest of the subsection is entirely dedicated to proving the following result:
\begin{theorem}\label{thm: there exists extoic ideal completions of a hyperbolic groupoid with appropriate growth conditions}
	Let $\G$ be a metrically hyperbolic groupoid such that the canonically associated length function is conditionally negative definite. Suppose $\mu$ is an invariant measure on $\G^{(0)}$ with full support, and that there exist ${R^{\prime}} > 1$ and $D > 0$ such that for any $k \in \N$, $\inf_{u \in \G^{(0)}} \left| \left\lbrace \gamma \in \G^{u} \colon l_S (\gamma) \leq k \right\rbrace \right|  \geq D {{R^{\prime}}^k}$. Then for every $2 \leq q < p < \infty$, the canonical $*$-surjection $$C_{\mathcal{L}_{\mu}^{p} (\G ), \mu}^{\ast} (\G) \twoheadrightarrow C_{\mathcal{L}_{\mu}^{q} (\G ), \mu}^{\ast} (\G) ,$$ is non-injective. 
\end{theorem}
The proof of the above theorem is inspired by Okayasu's proof in \cite{Okayasu:FreeGroupC*-algebrasAssociatedWithLP} of the existence of exotic group $\rm C^*$-algebras associated with the non-abelian free group on finitely many generators. We divide the proof into several smaller digestible lemmas, from which the theorem will follow. 

Before we start, let us state our assumptions and briefly discuss the notation we will use for the remainder of the section. Throughout, we assume that $\G$ is a metrically hyperbolic groupoid with compact open generating set $S$ such that $l_S$ is conditionally negative definite, that there exists a constant $\delta > 0$ for which all Cayley graphs $\G^u$ are $\delta$-hyperbolic, and that there is an invariant measure $\mu$ on the unit space with full support; we may without loss of generality assume that $\mu$ is moreover a probability measure. When we say that $\gamma$ has length $k$, we just mean that $l_S (\gamma) = k$. By $W_k$ we mean the set of elements in $\G$ with length $k$, while $B_k$ denotes the set of all elements in $\G$ which has length less than or equal to $k$. By $\chi_k$ we mean the characteristic function of the set $W_k$.
We assume, moreover, that $\G$ satisfies the above growth condition; so there exist ${R^{\prime}} > 1$ and $D > 0$ such that for any $k \in \N$, $$\inf_{u \in \G^{(0)}} \left| B_k \cap \G^u \right|  \geq D {{R^{\prime}}^k}.$$ Since $S$ is compact and $\G$ is étale, there exists finitely many bisections covering $S$, and thus it follows that there exists $R > 1$ such that for all $k \in \N_0$, we have $\sup_{u \in \G^{(0)}} | W_k \cap \G^u | \leq R^k$. Let $\mathcal{L}_{\mu}^{p}(\G)$ denote the ideal in the preceding paragraph. For $f \in \mathcal{L}_{\mu}^{p} (\G)$, and $p \in [1, \infty)$ we denote by $$|f|_{p}^{p} := \int_{\G^{(0)}} \sum_{x \in \G^{u}} |f(x)|^p \, d \mu(u) = \int_{\G^{(0)}} \sum_{x \in \G_{u}} |f(x)|^p \, d \mu(u) ,$$ while as usual $$\| f \|_{\mathcal{L}_{\mu}^{p}(\G) ,\mu} = \sup \left\lbrace \| \pi_\mu (f) \| \colon \pi \text{ is an } \mathcal{L}_{\mu}^{p} (\G) \text{-representation} \right\rbrace .$$ Given any $f \in B(\G)$, the supremum-norm will be denoted by $\| f \|_\infty = \sup_{\gamma \in \G} |f(\gamma)|$, while for any $u \in \G^{(0)}$, $|f|_{\ell^1 (\G^u)}$ denotes the usual $\ell^1$-norm on $\G^u$; so $|f|_{\ell^1 (\G^u)} = \sum_{x \in \G^u} |f(x)|$.

The following lemma is inspired by \cite[Proposition 4.3]{OzawaRieffel:HyperbollicGroupC*-Algebras} which in turn provides a detailed argument of \cite[Chapter 3.5, Lemma 4 and Theorem 5]{Connes:NoncommutativeGeometry}.
\begin{lemma} \label{lem: first techincal lemma in exotic hyperbolic groupoid Cstar algebras}
	Let $k,n \in \N_0$ and fix $f, g \in C_c (\G)$ such that $\supp (f) \subset W_k$, $\supp (g) \subset W_n$ and $|g| \leq 1$. Fix any $u \in \G^{(0)}$. Then for all $|k-n| \leq m \leq n+k$, we have that $$ |(f \ast g) \chi_m |_{\ell^1 (\G^u)} \leq C |f|_{\ell^1(\G^u)} , $$ where $C$ is a constant only depending on the hyperbolicity constant $\delta > 0$. For all other $m$, $(f \ast g) \chi_m = 0$.
	\begin{proof}
		Let us first show that $(f \ast g) \chi_{m} \neq 0 $ only if $|k-n| \leq m \leq k+n$. Note that $$ f \ast g (x) = \sum_{y z = x} f({y}) g (z) = \sum_{\substack{yz = x \\ l_S (y) = k \\ l_S (z) = n}} f (y) g (z). $$ Thus, if $f \ast g (x)$ is to be nonzero, we need $l_S (x) \leq n + k$, showing the upper bound on $m$. Fix $x$ of length $m$; then if $x = y z$ with $l_S (y) = k$ and $l_S (z) = n$, we have that $$k = l_S (y) \leq l_S (y z) + l_S (z^{-1}) = l_S(x) + l_S(z) = m + n,$$ and so $m \geq k-n$. One shows analogously that $m \geq n-k$. Thus, we need $m \geq |k-n|$, and we may therefore conclude that $(f \ast g) \chi_m \neq 0$ only if $|k-n| \leq m \leq k+n$.
		
		Let $x \in \G^{u}$ have length $m$, where $m$ satisfies $|k-n| \leq m \leq k+n$. Let $p := k+n-m$, and put $q := p/2$ if $p$ is even, and $q := (p-1)/2$ if $p$ is odd. Let $\tilde{q} = p-q$ so that $q + \tilde{q} = p$, and notice that $q \leq \tilde{q} \leq q + 1$, $k \geq q$ and $n \geq \tilde{q}$. Fix a decomposition of $x$ of products of $m$ elements of $S$, and let $\bar{x}$ and $\tilde{x}$ respectively denote the first $k-q$ elements and last $n- \tilde{q}$ elements. Thus, $x = \bar{x} \tilde{x}$, and we therefore need $l_S (\bar{x}) = k-q$ and $l_S (\tilde{x}) = n-\tilde{q}$. Given $n,k$ and $q$, such a decomposition is not unique, but we nevertheless fix one for each $x$ of length $m$. 
		
		Let now $x = yz$ for $y$ and $z$ of length $k$ and $n$ respectively. Then by the $\delta$-hyperbolicity of the Cayley graph $\G^{r(x)}$, we obtain that
		$$ d(\bar{x},y) + d(x,r(x)) \leq \max\{ d(\bar{x}, r(x)) + d(y,x), d(\bar{x},x) + d(y, r(x)) \} + \delta .$$ Since $d(\bar{x}, r(x)) = l_S (\bar{x}) = k-q$, $d(y,x) = l_S (y^{-1} x) = l_S (z) = n$, $d (y, r(x)) = k$ and $d(\bar{x},x) = l_S (\bar{x}^{-1}x) = l_S (\tilde{x}) = n- \tilde{q}$, we get that 
		$$ d(\bar{x},y) + d(x,r(x)) \leq k+n - q + \delta .$$ Since $d(x, r(x)) = m$ and $m = k+n-p = k+n-q-\tilde{q}$, we obtain that $$ d(\bar{x},y) \leq k+n-q-m + \delta = \tilde{q} + \delta .$$ Putting $\xi := \bar{x}^{-1}y$, we see that $l_S (\xi) = d(\bar{x},y) \leq \tilde{q} + \delta$ and that $\bar{x} \xi = y$ and $\xi z = \bar{x}^{-1} yz = \bar{x}^{-1} \bar{x} \tilde{x} = \tilde{x}$, so $z = \xi^{-1} \tilde{x}$. Thus, we have that 
		\begin{align*}
			|(f \ast g)\chi_m |_{\ell^1 (\G^u)} &= \sum_{\substack{x \in \G^{u} \\ l_S (x) = m}} |f \ast g(x)| \\
			&= \sum_{\substack{x \in \G^{u} \\ l_S (x) = m}} \big| \sum_{\substack{yz = x \\ l_S (y) = k \\ l_S (z) = n}} f (y) g (z) \big| \\
			&\leq \sum_{\substack{x \in \G^{u} \\ l_S (x) = m}} \sum_{\substack{yz = x \\ l_S (y) = k \\ l_S (z) = n}} |f (y)| |g (z)| \\
			&\leq \sum_{\substack{x \in \G^{u} \\ l_S (x) = m}} \sum_{\substack{\xi \in \G^{s(\bar{x})} \\ l_S (\xi) \leq \tilde{q} + \delta}} |f (\bar{x} \xi)| |g (\xi^{-1} \tilde{x})| \\
			&\leq \sum_{\substack{x \in \G^{u} \\ l_S (x) = m}} \sum_{\substack{\xi \in \G^{s(\bar{x})} \\ l_S (\xi) \leq \tilde{q} + \delta}} |f (\bar{x} \xi)| ,
		\end{align*}
		where we have used in the last step that $|g| \leq 1$.
		
		Let us estimate the number of decompositions $y = \xi t = \eta s$ such that $l_S (y) = k$, $l_S (s), l_S (t) \leq \tilde{q} + \delta$ and $l_S (\xi) = l_S (\eta) = k-q$. For such $y, \xi, \eta, t$ and $s$ as above, we have by $\delta$-hyperbolicity of the Cayley graphs that $$ d(\xi, \eta) + d(r(y), y) \leq \max \{ d(\xi, r(y)) + d(\eta ,y) , d(\xi , y) + d (\eta, r(y)) \} + \delta .$$ But $d(r(y), y) = l_S (y) = k$, $d(\xi, r(y)) = l_S (\xi) = k-q$, $d(\eta ,y) = l_S (\eta^{-1} y) = l_S (s) \leq \tilde{q} + \delta$, $d(\xi , y) = l_S (t) \leq \tilde{q} + \delta$ and $d (\eta, r(y)) = l_S (\eta) = k-q$. Thus $$ d(\xi, \eta) + k = d(\xi, \eta) + d(r(y), y) \leq k-q + \tilde{q} + \delta + \delta ,$$ so that $ d(\xi, \eta) \leq 2 \delta + 1 $. Let $C$ denote the largest number of elements within a ball of radius $2 \delta + 1$ in the Cayley graphs of $\G$. This constant exists because we may cover the compact generating set $S$ by finitely many bisections, and then the vertices in all of the Cayley graphs have uniformly bounded degree. It follows that for any $y \in \G^{u}$ of length $k$, the number of decompositions as above is at most $C$. Thus,
		$$ |(f \ast g)\chi_m |_{\ell^1 (\G^u)} \leq \sum_{\substack{x \in \G^{u} \\ l_S (x) = m}} \sum_{\substack{\xi \in \G^{s(\bar{x})} \\ l_S (\xi) \leq \tilde{q} + \delta}} |f (\bar{x} \xi)| \leq C \sum_{\substack{y \in \G^u \\ l_S (y) = k}} |f(y)| = C |f|_{\ell^1 (\G^u)}.$$ This proves the lemma.
	\end{proof}
\end{lemma}

\begin{lemma} \label{lem: second technical lemma for hyperbolic groupoid C*-algebras}
	Let $k \in \N$ and $\alpha \in (0,1)$. Let $1 < q \leq 2 \leq p < \infty$ with $\frac{1}{q} + \frac{1}{p} = 1$. If $\pi$ is an $\mathcal{L}_{\mu}^{p} (\G)$-representation, and $f(x) := \alpha^{l_S (x)} \chi_k (x)$, then $$ \|\pi_\mu (f)\| \leq 2C (k+1) | f |_q .$$
	\begin{proof}
		Since $\mu$ is invariant and $\pi$ is an $\mathcal{L}_{\mu}^{p}$-representation, one can prove in an analogous way as in \cref{lem: Bc-representations are weakly contained in regular reps.} with Hölder's inequality instead of Cauchy-Schwartz, that for any $g \in C_c (\G)$, we have $$\left\| \pi_\mu (g) \right\| \leq \liminf_{n \to \infty} \left| (g^{\ast} \ast g)^{\ast 2n} \right|_{q}^{1/4n}.$$ In particular, for $g = f = f^{\ast} = \alpha^{l_S (\cdot)} \chi_k$, we have that $$\left\| \pi_\mu (f) \right\| \leq \liminf_{n \to \infty} \left| f^{\ast 4n} \right|_{q}^{1/4n}.$$ Moreover, 
		\begin{align*}
			\left|  f^{\ast 4n} \right|_{q}^{1/4n} &= \left( \int_{\G^{(0)}} \sum_{x \in \G^{u}} \left|f^{\ast 4n} (x) \right|^q \, d \mu(u) \right)^{1/4nq} \\
			&= \alpha^{k} \left( \int_{\G^{(0)}} \sum_{x \in \G^{u}} \left|\chi_{k}^{\ast 4n} (x) \right|^q \, d \mu(u) \right)^{1/4nq} 	\\
			&\leq  \alpha^{k} \left( \int_{\G^{(0)}} \left( \sum_{x \in \G^{u}} \chi_{k}^{\ast 4n} (x) \right)^q \, d \mu(u) \right)^{1/4nq}  \\
			&= \alpha^{k} \left( \int_{\G^{(0)}} \left| \chi_{k}^{\ast 4n} \right|_{\ell^1 (\G^u)}^{q} \, d \mu(u) \right)^{1/4nq} .
		\end{align*}
		Our aim is to show that $$ \liminf_{n \to \infty} \left( \int_{\G^{(0)}} \left| \chi_{k}^{\ast 4n} \right|_{\ell^1 (\G^u)}^{q} \, d \mu(u) \right)^{1/4nq} \leq 2C (k+1) .$$ If this is true, then indeed $$ \left\| \pi_\mu (f) \right\| \leq 2C (k+1) \alpha^{k} \leq 2C (k+1) \alpha^{k} | \chi_k |_{q} = 2C (k+1) | f |_q ,$$ where we have simply used that $|\chi_k|_q \geq 1$ for any $k \in \N_0$.
		
		We proceed much as in \cite[Lemma 1.4]{Haagerup:AnExampleOfANonnnuclearC*-algebrasWhichHasTheMetricApproximationProperty}. Put $g = \chi_{k}^{\ast (4n-1)} $, so that $\chi_{k}^{\ast 4n} = g \ast \chi_k$. By definition, $\supp(\chi_k) = W_k$, and we also have that $g \in C_c (\G)$. Write $g_i = g \chi_i$ so that $\supp (g_i) \subset W_i$. Fix any unit $u \in \G^{(0)}$. Then $|g |_{\ell^1 (\G^u)} = \sum_{i = 0}^{\infty} |g_i |_{\ell^1 (\G^u)}$. The sum is finite since $g_i = 0$ for all but finitely many $i$. Put $h := \chi_{k}^{\ast 4n} = g \ast \chi_k = \sum_{i = 0}^{\infty} g_i \ast \chi_k$, and put $h_m = h \chi_m$, so that $|h|_{\ell^1 (\G^u)} = \sum_{m = 0}^{\infty} |h_m |_{\ell^1 (\G^u)}$. Again, the sum is finite because $h_m$ is zero for all but finitely many $m$. By \cref{lem: first techincal lemma in exotic hyperbolic groupoid Cstar algebras}, we have that $$|(g_i \ast \chi_k)\chi_{m} |_{\ell^1 (\G^u)} \leq C |g_i|_{\ell^1 (\G^u)},$$ when $| k - i | \leq m \leq k + i$, and $(g_i \ast \chi_k)\chi_{m} = 0$ for all other $m$. From $| k - i | \leq m \leq k + i$, it follows that $| k - m | \leq i \leq k + m$. Hence
		\begin{align*}
			|h_m |_{\ell^1 (\G^u)} &= \left| \sum_{i = 0}^{\infty} (g_i \ast \chi_k) \chi_m \right|_{\ell^1 (\G^u)} \leq \sum_{i=0}^{\infty} |(g_i \ast \chi_k)\chi_m |_{\ell^1 (\G^u)} \leq C \sum_{i = |m-k|}^{m+k} |g_i|_{\ell^1 (\G^u)}.
		\end{align*}
		We may write any $|m-k| \leq i \leq m+k$ as $i = m+k-j$, for some $0 \leq j \leq m+k$. In fact, $j$ satisfies the inequalities $0 \leq j \leq \min\{2m,2k\}$, because $|m-k| \leq i = m+k-j$ implies both $j \leq 2m$ and $j \leq 2k$. Using this, we obtain
		\begin{align*}
			|h_m |_{\ell^1 (\G^u)} \leq C \sum_{j = 0}^{\min\{2m,2k\}} |g_{m+k-j}|_{\ell^1 (\G^u)} 
		\end{align*}
		It follows that
		\begin{align*}
			\left|g \ast \chi_k \right|_{\ell^1 (\G^u)} &= |h|_{\ell^1 (\G^u)} \\
			&= \sum_{m = 0}^{\infty} |h_m |_{\ell^1 (\G^u)} \\
			&\leq C \sum_{m = 0}^{\infty} \sum_{j = 0}^{\min\{2m,2k\}} |g_{m+k-j}|_{\ell^1 (\G^u)} \\
			&= C \sum_{j = 0}^{2k} \sum_{m = j}^{\infty} |g_{m+k-j}|_{\ell^1 (\G^u)} \\
			&= C \sum_{j = 0}^{2k} \sum_{m = k}^{\infty} |g_{m}|_{\ell^1 (\G^u)} \\
			&\leq C \sum_{j = 0}^{2k} |g|_{\ell^1 (\G^u)} = C (2k+1) |g|_{\ell^1 (\G^u)}.
		\end{align*}
		That is, $$ \left|\chi_{k}^{\ast 4n} \right|_{\ell^1 (\G^u)} \leq C (2k+1) \left| \chi_{k}^{\ast (4n-1)} \right|_{\ell^1 (\G^u)} .$$ By induction, we get $$ \left|\chi_{k}^{\ast 4n } \right|_{\ell^1 (\G^u)} \leq (C (2k+1))^{4n - 1} \left|\chi_k \right|_{\ell^1 (\G^u)}.$$ 
		It follows that 
		\begin{align*}
			&\liminf_{n \to \infty} \left( \int_{\G^{(0)}} \left| \chi_{k}^{\ast 4n} \right|_{\ell^1 (\G^u)}^{q} \, d \mu(u) \right)^{1/4nq} \\
			&\leq \liminf_{n \to \infty} (C(2k+1))^{\frac{4n-1}{4n}} \left( \int_{\G^{(0)}} \left| \chi_k \right|_{\ell^1 (\G^u)}^{q} \, d \mu(u) \right)^{1/4nq} \\
			&= \lim_{n \to \infty} (C(2k+1))^{\frac{4n-1}{4n}} \left( \int_{\G^{(0)}}\left| \chi_k \right|_{\ell^1 (\G^u)}^{q} \, d \mu(u) \right)^{1/4nq} \\
			&= C(2k+1) \leq 2C (k+1).
		\end{align*}
		As mentioned before, it follows from this that 
		$$ \left\| \pi_\mu (f) \right\| \leq 2C (k+1) \alpha^{k} \leq 2C (k+1) \alpha^{k} \left| \chi_k \right|_{q} = 2C (k+1) \left| f \right|_q ,$$ simply using that $\left|\chi_k \right|_q \geq 1$.
	\end{proof}
\end{lemma}

Given any bounded Borel map $\phi \colon \G \to \C$, we may define a linear functional on $C_c (\G)$ using the invariant measure $\mu$, by defining $\omega_\phi (f) = \int_{\G} f \phi \, d \nu$, where $\nu$ is the induced measure from $\mu$. We will distinguish the $\rm C^*$-algebras $C_{\mathcal{L}_{\mu}^{p} (\G), \mu}^{*}(\G)$, for $p \in [2, \infty)$, by determining which of the continuous positive definite functions of the form $\phi_\alpha (x) = \alpha^{l_S (x)}$, where $\alpha \in (0,1)$, the linear functional $\omega_{\phi_\alpha}$ can and cannot be extended to a state on $C_{\mathcal{L}_{\mu}^{p} (\G), \mu}^{*}(\G)$.

\begin{lemma} \label{lem: technical lemma 6 concerning Lp-representations from positive definite Lp functions}
	Let $\alpha \in (0,1)$ and $p \in [1, \infty)$. If $\phi_\alpha$ is in $\mathcal{L}_{\mu}^{p} (\G)$, then the corresponding GNS representation is an $\mathcal{L}_{\mu}^{p} (\G)$-representation.
	\begin{proof}
		Let $\pi$ be the corresponding GNS representation as in \cref{prop: positive definite function gives rise to a unitary representation}. Let $f,g \in C_c (\G)$ be sections coming from a fundamental sequence for the associated Borel Hilbert bundle, and put $\psi (x) := \langle \pi (x) \sigma(f) , \sigma(g) \rangle$. We aim to show that $\psi$ is in $\mathcal{L}_{\mu}^{p} (\G)$. Recall from \cref{lem: GNS representation associated with positive definite B_0 function is B_0 representation} that actually $$ \psi (x) =  \overline{g} \ast \phi_\alpha \ast \overline{f}^{\ast} (x) ,$$ for $x \in \G $, and as remarked there, it suffices to show that $g \ast \phi_\alpha \ast f$ is in $\mathcal{L}_{\mu}^{p} (\G)$ for $f$ and $g$ supported on a bisection. In this case, notice that given any $x \in \G$, if there exists a $z \in \G_{s(x)}$ such that $f(z ) \neq 0$, then this $z$ is unique. Similarly, if there exists $y \in \G^{r(x)}$ such that $g(y) \neq 0$, then this $y$ is unique. In what follows, we shall denote these unique elements elements by $y_{r(x)}$ and $z_{s(x)}$. We have that
		\begin{align*}
			&\int_{\G^{(0)}} \sum_{x \in \G^u} | g \ast (\phi_\alpha \ast f) (x) |^p \, d \mu(u) \\ 
			&\leq \int_{\G^{(0)}} \sum_{x \in \G^u} \left( \sum_{y \in \G^u} | g (y)| |\phi_\alpha \ast f (y^{-1} x)| \right)^p \, d \mu(u) \\
			&\leq \int_{r(\supp(g))} \sum_{x \in \G^u} | g (y_u)|^p |\phi_\alpha \ast f (y_{u}^{-1} x)|^p  \, d \mu(u)  \\
			&\leq \int_{r(\supp(g))} \sum_{x \in \G^u \cap \G_{s(\supp(f))}} |g(y_u)|^p \left( \sum_{z \in \G_{s(x)}} |\phi_\alpha (y_{u}^{-1} x z^{-1}) | |f(z)| \right)^p  \, d \mu(u) \\
			&= \int_{r(\supp(g))} \sum_{x \in \G^u \cap \G_{s(\supp(f))}} |g(y_u)|^p | \phi_\alpha (y_{u}^{-1} x z_{s(x)}^{-1}) |^p |f(z_{s(x)})|^p  \, d \mu(u) \\
			&\leq \| g \|_{\infty}^{p} \| f \|_{\infty}^{p} \int_{r(\supp(g))} \sum_{x \in \G^u \cap \G_{s(\supp(f))}} | \phi_\alpha (y_{u}^{-1} x z_{s(x)}^{-1}) |^p \, d \mu(u).
		\end{align*}
		For any composable ordered triple $(y,x,z) \in \G^{(3)}$, we have $ l_S (x) \leq l_S (y^{-1}) + l_S (yxz) + l_S (z^{-1}) $. So, in the above calculations, if we put $M := \sup_{\gamma \in \supp(f) \cup \supp(g)} l_S (\gamma)$, then $l_S (x) \leq l_S (y_{u}^{-1} x z_{s(x)}^{-1}) + 2M ,$ and so $\phi_{\alpha} (x) \geq \alpha^{2M} \phi_{\alpha}(y_{u}^{-1} x z_{s(x)}^{-1})$. Therefore,
		\begin{align*}
			&\int_{\G^{(0)}} \sum_{x \in \G^u} | g \ast (\phi_\alpha \ast f) (x) |^p \, d \mu(u) \\
			&\leq \| g \|_{\infty}^{p} \| f \|_{\infty}^{p} \int_{r(\supp(g))} \sum_{x \in \G^u \cap \G_{s(\supp(f))}} | \phi_\alpha (y_{u}^{-1} x z_{s(x)}^{-1}) |^p \, d \mu(u) \\
			&\leq \| g \|_{\infty}^{p} \| f \|_{\infty}^{p} \alpha^{-2Mp} \int_{\G^{(0)}} \sum_{x \in \G^u} | \phi_\alpha(x) |^p \, d \mu(u) < \infty.
		\end{align*}
		This proves the lemma.
	\end{proof}
\end{lemma}

\begin{lemma} \label{lem: characterization of when the positive definite functions in question can be extended to the hyperbolic groupoid ideal completions}
	Let $2 \leq p < \infty$ and fix $\alpha \in (0,1)$. The following are equivalent:
	\begin{itemize}
		\item[(1)] $\omega_{\phi_\alpha}$ can be extended to a state on $C_{\mathcal{L}_{\mu}^{p} (\G), \mu}^{*}(\G)$;
		\item[(2)] $\sup_{k \in \N} | \phi_\alpha \chi_k |_p (k+1)^{-1} < \infty$;
		\item[(3)] The map $x \mapsto \phi_\alpha(x) (1 + l_S (x))^{-1-2/p}$ belongs to $\mathcal{L}_{\mu}^{p} (\G)$;
		\item[(4)] For any $\beta \in (0,1)$, the map $x \mapsto \phi_\alpha(x) \beta^{l_S (x)} = \phi_{\alpha \beta} (x)$ belongs to $\mathcal{L}_{\mu}^{p} (\G)$.
	\end{itemize}
	\begin{proof}
		The proof is similar to that of \cite[Theorem 3.1]{Okayasu:FreeGroupC*-algebrasAssociatedWithLP}. 
		
		$(1) \implies (2)$: Assume $\omega_{\phi_\alpha}$ can be extended to a state on $C_{\mathcal{L}_{\mu}^{p} (\G), \mu}^{*}(\G)$. Thus, for all $f \in C_c (\G)$, we have $| \omega_{\phi_\alpha} (f) | \leq \|f\|_{\mathcal{L}_{\mu}^{p} (\G),\mu}$. Put $f := {\phi_\alpha}^{p-1} \chi_k$, for a fixed $k \in \N$. Then $f \in C_c (\G)$ and $| \omega_{\phi_\alpha} (f) | = | \phi_\alpha \chi_k |_{p}^{p}$. If $\sigma$ is any $\mathcal{L}_{\mu}^{p} (\G)$-representation, then by \cref{lem: second technical lemma for hyperbolic groupoid C*-algebras}, we have that $$ \left\| \sigma_\mu (f) \right\| \leq 2C (k+1) | f |_q . $$ Therefore, $\|f\|_{\mathcal{L}_{\mu}^{p} (\G),\mu} \leq 2C (k+1) |f|_{q}$, where $q$ is the Hölder conjugate of $p$, and so $$ |\phi_\alpha \chi_k |_{p}^{p} = | \omega_{\phi_\alpha} (f) | \leq \|f\|_{\mathcal{L}_{\mu}^{p} (\G),\mu} \leq 2 C (k+1) |f|_{q} = 2 C (k+1) |\phi_\alpha \chi_k |_{p}^{p-1}.$$ That is, $| \phi_\alpha \chi_k |_p \leq 2 C (k+1)$, so that $\sup_{k} | \phi_\alpha \chi_k |_p (k+1)^{-1} < \infty$.
		
		$(2) \implies (3)$: Assume  $\sup_{k} | \phi_\alpha \chi_k |_p (k+1)^{-1} < \infty$. We have that
		\begin{align*}
			\int_{\G} |\phi_\alpha(x) (1+ l_S (x))^{-1 - 2/p}|^p \, d \nu (x) &= \int_{\G^{(0)}} \sum_{x \in \G^u} | \phi_\alpha(x)|^p (1 + l_S (x))^{-p-2} \, d \mu (u) \\
			&\leq \sum_{k = 0}^{\infty} \int_{\G^{(0)}} \sum_{x \in \G^u } |\phi_\alpha (x)|^p (1 + k)^{-p-2} \chi_k (x) \, d \mu(u) \\
			&= \sum_{k = 0}^{\infty} (1 + k)^{-2} \int_{\G^{(0)}} \sum_{x \in \G^u } \big|\phi_\alpha(x) (1 + k)^{-1} \chi_k (x) \big|^p \, d \mu(u) \\
			&= \sum_{k = 0}^{\infty} (1+k)^{-2} \big( |\phi_\alpha \chi_k|_p (1 + k)^{-1} \big)^{p} \\
			&\leq \big( \sup_k | \phi_\alpha \chi_k |_p (1+k)^{-1} \big)^{p} \sum_{k = 0}^{\infty} \frac{1}{(k+1)^2} < \infty.
		\end{align*}
		
		$(3) \implies (4)$: Since $\beta \in (0,1)$, we have that $\beta^{k} < (1+ k)^{-1 -2/p}$ for all sufficiently large $k \in \N$. The implication follows from this.
		
		$(4) \implies (1)$: Let $\sigma_\beta$ be the GNS representation corresponding to $\phi_{\alpha \beta}$ with bounded Borel section $\xi_{\beta}$ satisfying $$ \phi_{\alpha \beta} (x) = \langle \sigma_\beta (x) \xi_\beta (s(x)) , \xi_\beta (r(x)) \rangle_{H_{\sigma_\beta} (r(x))},$$ for $x \in \G$. Since $$1 = | \phi_{\alpha \beta} (u) |^2 = \|\xi_\beta (u)\|_{H_{\sigma_\beta} (u)}^{2},$$ for any $u \in \G^{(0)}$, it follows that $\xi_\beta$ is a unit vector in $H_{\sigma_\beta, \mu}$. Moreover, by \cref{lem: technical lemma 6 concerning Lp-representations from positive definite Lp functions}, $\sigma_\beta$ is an $\mathcal{L}_{\mu}^{p} (\G)$-representation. It follows that for every $\beta \in (0,1)$, the map given for any $f \in C_c (\G)$ by $$\omega_{\phi_{\alpha \beta}} (f) = \int_{\G} f \phi_{\alpha \beta} \, d\nu = \int_{\G} f \phi_{\alpha \beta} \, d\nu_0 = \langle (\sigma_{\beta})_\mu (f) \xi_{\beta}, \xi_{\beta} \rangle ,$$ can be extended to a state on $C_{\mathcal{L}_{\mu}^{p} (\G),\mu}^{*}(\G)$. Above we have used that the measure $\mu$ is invariant, so $d \nu = d \nu^{-1} = d \nu_0$. Since $\lim_{\beta \to 1} \phi_{\alpha \beta} = \phi_{\alpha}$ uniformly on compact sets, the Lebesgue dominated convergence theorem gives that for any $f \in C_c (\G)$, we have $$ \omega_{\phi_{\alpha}} (f) = \int_{\G} f \phi_{\alpha} \, d \nu = \lim_{\beta \to 1} \int_{\G} f \phi_{\alpha \beta} \, d \nu = \lim_{\beta \to 1} \omega_{\phi_{\alpha \beta}} (f) .$$ In particular, $$ | \omega_{\phi_{\alpha}} (f) | = \lim_{\beta \to 1} | \omega_{\phi_{\alpha \beta}} (f) | \leq \|f\|_{\mathcal{L}_{\mu}^{p} (\G), \mu} ,$$ for every $f \in C_c (\G)$, and therefore, $\omega_{\phi_{\alpha}}$ can be extended to a state on $C_{\mathcal{L}_{\mu}^{p} (\G),\mu}^{*}(\G)$.
	\end{proof}
\end{lemma}

Recall that since $S$ is a generating set and $\G$ is an étale groupoid, we may cover $S$ by a finite number of bisections, say $R$, and then it follows that for any $k \in \N_0$, we have $\sup_{u \in \G^{(0)}} |W_k \cap \G^u| \leq R^k$. Recall also that we are assuming a lower bound as well, namely that $\inf_{u \in \G^{(0)}} |B_k \cap \G^{u}| \geq D {R^{\prime}}^k$. 

\begin{proof}[Proof of \cref{thm: there exists extoic ideal completions of a hyperbolic groupoid with appropriate growth conditions}]
	Let $R^{\prime}$ and $R$ be as above. First of all, since the invariant measure $\mu$ has full support and $B_c (\G) \subset \mathcal{L}_{\mu}^{p} (\G)$, it follows by \cref{prop: reduced C*-algebra is the one associated with the ideal B_c (G)} and \cref{lem: ideal containment induces quotient map of c*-algebras} that $$ \|\cdot\|_{\mathcal{L}_{\mu}^{p} (\G),\mu} \geq \|\cdot\|_{B_c (\G) , \mu} = \|\cdot\|_{r},$$ for any $p \in [2, \infty)$. According to \cref{def: exotic groupoid C*-algebra}, $C_{\mathcal{L}_{\mu}^{p} (G), \mu}^{\ast}(\G)$ is a groupoid $\rm C^*$-algebra. 
	
	Now let $\alpha \in (0,1)$ be given. Then
	\begin{align*}
		\int_{\G} \phi_\alpha \, d \nu &= \int_{\G^{(0)}} \sum_{x \in \G^u} \phi_\alpha (x) \, d \mu(u) \\
		&= \int_{\G^{(0)}} \sum_{k = 0}^{\infty} |W_k \cap \G^u| \alpha^k \, d \mu(u) \\
		&\leq \sum_{k = 0}^{\infty} (R \alpha)^k < \infty ,
	\end{align*}
	if $\alpha < 1/R$. That is, $\phi_\alpha \in \mathcal{L}_{\mu}^{1} (\G)$ whenever $\alpha < 1/R$. At the same time, for any $N \in \N$, we have 
	\begin{align*}
		&\int_{\G^{(0)}} \sum_{k = 0}^{N} |W_k \cap \G^u | \alpha^{k} \, d \mu(u) \\
		&= 1 + \int_{\G^{(0)}} \sum_{k = 1}^{N} (|B_k \cap \G^u | - |B_{k-1} \cap \G^u |) \alpha^{k} \, d \mu(u) \\
		&\geq 1 - \alpha + \int_{\G^{(0)}} \sum_{k = 1}^{N} |B_k \cap \G^u | (\alpha^{k} - \alpha^{(k+1)}) \, d \mu(u) \\
		&= 1 - \alpha + (1 - \alpha) \int_{\G^{(0)}} \sum_{k = 1}^{N} |B_k \cap \G^u | \alpha^{k} \, d \mu(u) \\
		&\geq 1 - \alpha + (1 - \alpha) D \sum_{k = 1}^{N} (R^{\prime} \alpha)^{k}.
	\end{align*}
	So, if $\alpha \geq 1/R^{\prime}$, then $\phi_{\alpha} \notin \mathcal{L}_{\mu}^{1} (\G)$, because for such $\alpha$, we have
	\begin{align*}
		D(1 - \alpha )N &\leq D(1 - \alpha) \sum_{k = 1}^{N} (R^{\prime} \alpha)^{k}  \\
		&\leq \int_{\G^{(0)}} \sum_{k = 0}^{N} |W_k \cap \G^u | \alpha^{k} \, d \mu(u) \\
		&\leq \int_{\G^{(0)}} \sum_{k = 0}^{\infty} |W_k \cap \G^u | \alpha^{k} \, d \mu(u) = \int_{\G} \phi_{\alpha} (x) \, d\nu(x),
	\end{align*}
	for all $N \in \N$. Due to the above analysis, we see that $$ \alpha_0 := \sup\{ \alpha \in (0,1) \mid \phi_\alpha \in \mathcal{L}_{\mu}^{1} (\G) \} ,$$ lies in $(0,1)$. Now let $p \geq 2$ be given. We claim that $\phi_{(\alpha_0 \beta)^{1/p}} \in \mathcal{L}_{\mu}^{q} (\G)$ for all $\beta \in (0,1)$ if and only if $q \geq p$; if this is true, then it follows by \cref{lem: characterization of when the positive definite functions in question can be extended to the hyperbolic groupoid ideal completions} that for $q \geq 2$, $\omega_{\phi_{{\alpha_0}^{1/p}}}$ extends to a state on $C_{\mathcal{L}_{\mu}^{q} (\G), \mu}^{\ast}(\G)$ if and only if $q \geq p$. Consequently, for every $2 \leq q < p < \infty$, the canonical $*$-surjection $$ C_{\mathcal{L}_{\mu}^{p} (\G), \mu}^{\ast}(\G) \twoheadrightarrow C_{\mathcal{L}_{\mu}^{q} (\G), \mu}^{\ast}(\G), $$ is not injective, proving the theorem. 
	
	To establish the claim, notice first that for every $\beta \in (0,1)$, we have that $$ \phi_{(\alpha_0 \beta)^{1/p}} \in \mathcal{L}_{\mu}^{q} (\G) \iff \phi_{(\alpha_0 \beta)^{q/p}} \in \mathcal{L}_{\mu}^{1} (\G).$$ Now if $q \geq p$, then $(\alpha_0 \beta)^{q/p} \leq \alpha_0 \beta < \alpha_0$, so that by definition of $\alpha_0$, we have that $\phi_{(\alpha_0 \beta)^{q/p}} \in \mathcal{L}_{\mu}^{1} (\G)$, or equivalently, $\phi_{(\alpha_0 \beta)^{1/p}} \in \mathcal{L}_{\mu}^{q} (\G)$. If $q < p$, then ${\alpha_0}^{q/p} > \alpha_0$, and there exists $\beta \in (0,1)$ such that $(\alpha_0 \beta)^{q/p} > \alpha_0$. Again, by definition of $\alpha_0$, we have that $\phi_{(\alpha_0 \beta)^{q/p}} \notin \mathcal{L}_{\mu}^{1} (\G)$, or equivalently, $\phi_{(\alpha_0 \beta)^{1/p}} \notin \mathcal{L}_{\mu}^{q} (\G)$.
\end{proof}

\begin{example} \label{ex: Non-elememtary hyperbolic transformation groups have exotic ideal completions}
	Let $\Gamma$ be a finitely generated non-elementary hyperbolic group such that the canonical length function is conditionally negative definite. Suppose that $\Gamma$ acts minimally and non-amenably on a compact Hausdorff space $X$, and suppose moreover that there is an invariant measure $\mu$ on $X$ with respect to this action. In \cite{HjortAndMolberg:FreeContinuousActionsOnZeroDimensionalSpaces}, Hjort and Molberg showed that any countably infinite discrete group admits a free action on the Cantor space which has an invariant probability measure, and in \cite{Elek:FreeMinimalActionsOfCountableGroupsWithInvariantProbabilityMeasures}, Elek improved upon this by showing that one can take the action to be minimal. Therefore, such triples of groups, spaces and actions do exist. By minimality of the action and invariance of the measure $\mu$ under the action, the measure $\mu$ has in fact full support. The transformation groupoid $\Gamma \rtimes X$ has compact open generating set $F \times X$, where $F \subset \Gamma$ is a finite symmetric set of elements generating the group $\Gamma$. The metric spaces given by the Cayley graphs of $\Gamma \rtimes X$ are isomorphic (as metric spaces) to the Cayley graph of $\Gamma$, hence are all hyperbolic with respect to the same hyperbolicity constant. Thus $\Gamma \rtimes X$ is a metrically hyperbolic groupoid. Moreover, it is easy to see with our assumptions that the length function associated with $F \times X$ is conditionally negative definite. By \cite{Coornaert:BoundOnBallsInHyperbolicGroup} (see also \cite{Valinuas:ANoteOnGrowthOfHyperbolicGroups}) there exist constants $R,D > 1$ such that for all $k \in \N$, we have $D^{-1} R^k \leq |B_k \cap \Gamma \times \{x\}| \leq D R^k$, for every $x \in X$. It follows by \cref{thm: there exists extoic ideal completions of a hyperbolic groupoid with appropriate growth conditions} that for every $2 \leq q < p < \infty$, the canonical surjection 
	$$ C_{\mathcal{L}_{\mu}^{p} (\Gamma \rtimes X), \mu}^{*}(\Gamma \rtimes X) \twoheadrightarrow C_{\mathcal{L}_{\mu}^{q} (\Gamma \rtimes X), \mu}^{*}(\Gamma \rtimes X) ,$$ is not injective, and consequently the collection $\{ C_{\mathcal{L}_{\mu}^{p} (\Gamma \rtimes X), \mu}^{*}(\Gamma \rtimes X) \}_{p \in (2, \infty)}$ form an uncountable collection of exotic groupoid $\rm C^*$-algebras associated to $\Gamma \rtimes X$.
\end{example}

\begin{example} \label{ex: exotic crossed product from actions of the free group on the sphere}
	Consider the special orthogonal group, $\mathrm{SO} (3, \R)$, viewed as a discrete group. It acts on the unit ball $B_1 \subseteq \R^3$ by matrix multiplication and it is well-known that the Lebesgue measure is invariant under this action, and has full support. Let $R_z$ be the rotation by $\arccos(1/3)$ about the $z$-axis and $R_x$ the rotation by $\arccos(1/3)$ about the $x$-axis. Then it is a standard fact that $\mathbb{F}_2 \cong \langle R_x , R_z \rangle \leq \mathrm{SO}(3,\R)$, and so we obtain an action of $\mathbb{F}_2$ on $B_1$ which admits an invariant full support measure. 
	Note that this action of $\mathbb{F}_2$ on $B_1$ is not free since for example $R_z$ fixes any point on the $z$-axis.  As in \cref{ex: Non-elememtary hyperbolic transformation groups have exotic ideal completions}, we see that the transformation groupoid $\mathbb{F}_2 \rtimes B_1$ is metrically hyperbolic and the canonically associated length function is conditionally negative definite. Therefore, it follows by \cref{thm: there exists extoic ideal completions of a hyperbolic groupoid with appropriate growth conditions} that $\{ C_{\mathcal{L}_{\mu}^{p} (\mathbb{F}_2 \rtimes B_1), \mu}^{*}(\mathbb{F}_2 \rtimes B_1) \}_{p \in (2, \infty)}$ form an uncountable collection of exotic groupoid $\rm C^*$-algebras associated with the transformation groupoid $\mathbb{F}_2 \rtimes B_1$.
\end{example}

\begin{remark}
	Taking $\Gamma$ to be a finitely generated non-elementary hyperbolic group such that the canonically associated length function is conditionally negative definite, and $X$ to be a single point, then the transformation groupoid is just the group $\Gamma$. Then \cref{ex: Non-elememtary hyperbolic transformation groups have exotic ideal completions} shows that the family $\{ C_{\ell^p (\Gamma)}^{*}(\Gamma) \}_{p \in (2, \infty)}$ forms an uncountable collection of different exotic group $\rm C^*$-algebras, generalizing Okayasu's main result in \cite{Okayasu:FreeGroupC*-algebrasAssociatedWithLP} to said groups, though this also follows from results by Samei and Wiersma in \cite{SameiAndWiersma:ExoticC*AlgebrasOfGeometricGroups}.
\end{remark}

\begin{remark}
	Taking $\Gamma$ to be the free group on two generators acting freely and minimally on a second-countable compact Hausdorff space such that there exists an invariant probability measure, then \cref{ex: Non-elememtary hyperbolic transformation groups have exotic ideal completions} together with \cref{thm: there exists extoic ideal completions of a hyperbolic groupoid with appropriate growth conditions} provides an alternative way to prove \cite[Theorem 3.6]{ExelPittsZarikian:ExoticIdealsInFreeTransformationGroupC*Algebras} when the action is minimal. 
\end{remark}
\printbibliography

\end{document}